\documentclass{article}

\usepackage[english]{babel}

\usepackage[letterpaper,top=2cm,bottom=2cm,left=3cm,right=3cm,marginparwidth=1.75cm]{geometry}

\usepackage{amsfonts}
\usepackage{amsmath}
\usepackage{amsthm}
\usepackage{amssymb}
\usepackage{graphicx}
\usepackage[colorlinks=true, allcolors=blue]{hyperref}
\usepackage{float}
\usepackage{booktabs}
\usepackage[table,dvipsnames]{xcolor}

\usepackage{subcaption}
\usepackage{comment}
\usepackage{authblk}

\usepackage{tikz}

\newcommand{\R}{\mathbb{R}}
\newcommand{\V}{\mathbf{V}}
\newcommand{\X}{\mathbf{X}}

\newcommand{\U}{\mathbf{U}}
\newcommand{\D}{\mathbf{D}}
\newcommand{\Q}{\mathbf{Q}}
\newcommand{\Po}{\mathbf{P}}
\newcommand{\pN}{\text{pN}}
\newcommand{\e}{\mathbf{e}}
\newcommand{\blambda}{\boldsymbol{\lambda}}

\newcommand{\Rintar}{R_{\text{int}}^{\text{ar}}}
\newcommand{\Rintpo}{R_{\text{int}}^{\text{po}}}
\renewcommand{\leq}{\leqslant}
\renewcommand{\geq}{\geqslant}
\newcommand{\Proj}{\text{Proj}}
\newcommand{\microm}{\mu\text{m}}
\newcommand{\rad}{\text{rad}}
\newcommand{\h}{\text{h}}
\newcommand{\phirot}{\phi_{\text{rot}}}

\newcommand{\norme}[1]{\left\Vert #1\right\Vert}

\newcommand{\Mov}[1]{Suppl. Mov.\;#1}

\newtheorem{prop}{Proposition}
\newtheorem{remark}{Remark}

\newcommand{\reviewerA}[1]{#1}
\newcommand{\reviewerB}[1]{#1}

\newcommand{\reviewerAA}[1]{#1}
 
\date{}

\title{From flocking to jamming in collective cell dynamics: a Vicsek-like model including contact forces}
\author[1,2]{Laurent Navoret}
\author[1,2]{Roxana Sublet}
\author[3]{Marcela Szopos}

\affil[1]{IRMA, UMR 7501 Université de Strasbourg et CNRS, 7 rue René Descartes, 67000 Strasbourg, France}
\affil[2]{Inria, IRMA, Université de Strasbourg, CNRS UMR 7501, 7 rue René Descartes, 67084 Strasbourg, France}
\affil[3]{Université Paris Cité, CNRS, MAP5, F-75006 Paris, France}

\begin{document}
\maketitle

\begin{abstract}
The goal of the present work is to propose an agent-based model that originally combines classical Vicsek-like polarity alignments and contact forces, as implemented in the framework developed by Maury and Venel in \cite{MauryVenel2011}. The description additionally incorporates velocity feedback on polarity and soft attraction-repulsion interactions. After carefully studying the well posedness of the \reviewerA{model}, we introduce a suitable discretization and perform an extensive range of numerical experiments to assess the impact of different modeling ingredients. The dynamical system is capable of recovering the order-disorder phase transition of the flock, as well as the jamming effect in high density regimes. As such, the developed framework can be seen as a promising theoretical tool that could contribute to improving the understanding of complex collective cell dynamics and emerging tissue flows.
\end{abstract}

\section{Introduction}

Collective dynamics in biological systems and emergence of patterns are ubiquitous in nature. In animal groups, collective movements were identified as key mechanisms to facilitate navigation, access to food resources or protection and defense \cite{koch2011collective, Vicsek2012}. Pedestrian displacement and collective interactions are of particular interest when dealing with crowd and traffic management \cite{MauryVenel2011,feliciani2022introduction,helbing2005self}. Cell movements are observed during different important phases of epithelial tissue remodeling, such as embryogenesis, or in pathological situations, such as cancer or wound healing \cite{friedl2009collective}. Hundreds or even thousands of cells move in a coordinate manner and tissue-scale flows emerge. However, the underlying mechanisms are still poorly understood and are the subject of active research.

The present contribution focuses on the modeling of collective cell dynamics. In recent years, various physical models have been proposed in parallel with biological experiments to analyze the emergence of such movements and to understand how they can result from the interactions between individual cells. Some of these models attempt to describe the dynamics directly at the macroscopic level. This is the case of active gel type models \cite{marchetti2013hydrodynamics}, where the collective motion results from the polarity vector field. Other approaches rather propose an individual-based description of the cellular dynamics: in the particle model, each cell is described by its position, its velocity and its interactions with its neighbors, while in vertex model, cells are described as polygons. Cell activity is described at a mesoscopic scale by introducing some active forces and possibly attaching polarity vectors or nematic directions to each particle, to summarize the internal cell structures. We refer to \cite{graner} for a recent review and a comparative description of such models and to \cite{Deutschetal2020} for the multiscale aspects and their biological relevance.

In the present work, we propose a particular model from the category of individual based models, namely a two-dimensional Vicsek-type approach that includes jamming effects. Specifically, the original Vicsek model is based on a pointwise particle description of the cellular dynamics, in which the velocity alignment is in competition with the angular diffusion \cite{vicsek1995novel}. A particular feature of this model is that it leads to a phase transition between disordered and ordered phases as the ratio between alignment and diffusion increases. We further refer to \cite{chate2020dry} and the references therein for a review on its properties. Considered as a minimal model for collective displacement, such an approach has been successfully used to depict cell flows~\cite{hakim2017collective}, with Vicsek alignment interactions either on the velocities or on polarity vectors. Many variants have been subsequently proposed to enrich the description of different phenomena, such as attraction-repulsion interactions \cite{chate2008modeling,degond2015macroscopic,tarle2015modeling}, or cell division \cite{tarle2015modeling}. We also emphasize that the large-scale behavior of the model is still the subject of an intense research activity: macroscopic version of Vicsek-type models have been derived in ~\cite{degond2008continuum,degond2013macroscopic, BriantetalSIAM2022}.

A key point for an accurate description of cell dynamics in such particle models is the question of properly taking into account the volume exclusion constraint, which expresses the fact that cells cannot overlap during their displacement. A commonly used strategy, sometimes called in the literature the ``soft'' contact model, is to add short-range repulsive forces between particles. Their description possibly incorporates a stiff component to account for some limit in the cell two-dimensional compressibility, expressed as a potential with a singular behavior at a finite distance of the cell center, see for instance~\cite{Dorsogna2006,carrillo2014explicit}. These approaches have proved to be quite effective in many situations,
but they might require fine-tuning of the modeling parameters and the stiffness induces strong constraints on the numerical discretization. Alternatively, contact forces can be incorporated in the model as proposed in~\cite{MauryVenel2011} for pedestrians dynamics, strategy identified in the literature as the ``hard'' contact model.
Each particle is assimilated to a disk and the contact forces result from the projection of the velocity onto the set of the so-called admissible velocities. As a consequence, the resulting vector field provides the velocities compatible with the non-overlapping constraint. Interestingly, it has been shown that models incorporating this feature are able to recover some jamming effect as the `Faster is Slower'' effect in the context of an evacuation scenario \cite{faure2015crowd, Redaetal2024faster}. Additionally, non-spherical particles can also be considered \cite{bloch2023convex,bloch2023new} in such a framework. However, even though some equivalent macroscopic models have been proposed, they are not able to recover all the microscopic jamming configurations, as recently discussed in \cite{maury2023defective}.

In this challenging context, we have previously proposed and validated against experimental data an original model combining a Vicsek-type description of the cell dynamics and a ``hard'' contact model~\cite{nature_phys}. The advantage of this strategy lies in the possibility to strongly enforce the geometric constraints, but as a counterpart, a more complex mathematical and numerical framework is required to handle weak differential inclusions and constraint optimization problems. In the present paper, we further develop an in-depth analysis of the properties of the model at the continuous level and a systematic numerical exploration of its behavior, with emphasis on the ability of the proposed model to recapitulate the phase transition from flocking to jamming. In addition, soft local attraction-repulsion interactions are here introduced to model elastic interactions between cells, as proposed in~\cite{graner}, since the amplitude of these interactions compared with the Vicsek ones could have a significant impact on the collective dynamics, passing from collective migration to elastic regimes. Numerical simulations are performed to assess this transition. At the discrete level, following the method introduced in \cite{collision_maury}, the effective velocities are computed by solving a projection problem after linearizing the constraints using the Uzawa algorithm. In the case of spherical particles, this process ensures that cells do not overlap at each time iteration of the algorithm. This method is subsequently coupled with a semi-implicit particle method introduced in \cite{numerical_polarity} for the particle Vicsek dynamics.

While we show that this model is able to recover order-disorder phase transition, similarly to the classical Vicsek model~\cite{vicsek1995novel,chate2020dry},  additional confinement is required to observe jamming effects. In~\cite{nature_phys}, cells were confined in self-generated surrounding acto-myosin cables. Here, cells are confined in a given bounded domain and cell-wall contact forces at the boundaries prevent them from escaping. With these interaction rules, cells are expected to rapidly stick to the border, and consequently make them difficult to move. Instead of adding specific boundary behaviors, we rather rely on a relaxation mechanism of the polarity towards the velocity direction and demonstrate its ability to unlock such a configuration. Note that the resulting model is similar to the one proposed in \cite{degond2015macroscopic}, except that the repulsion interactions are here treated as contact forces. Finally, to further explore how the geometrical and topological properties of the environment impact the collective dynamics, we add fixed obstacles inside the domain. The contact with these additional boundaries is incorporated in a straightforward manner in the hard contact framework we have previously developed. The present work thus makes a new
contribution to collective dynamics models that exhibit non-trivial topological states \cite{Shankar_2017,bowick2022symmetry,degond2022bulk}.

\reviewerB{\textbf{Aim of the paper.} To summarize, the goal of the paper is twofold: (i) first, to provide a thorough mathematical analysis of the well-posedness of the model combining Vicsek dynamics with contact forces, (ii) second, to numerically explore the impact of congestion on the emergence of flocking  or rotation motions in bounded domains with periodic or rigid boundaries. Specifically, we aim to identify the effects of the velocity feedback on polarity, of the shape of the domain and of the soft attraction-repulsion force on the overall dynamics.}

The outline of the paper is the following. In Section~\ref{sec:model_analysis}, we describe the new individual-based model we introduce in the present work and provide a well-posedness analysis of a regularized version of this model. In Section~\ref{sec:numericalmethod}, we introduce a semi-implicit numerical discretization of the model and the practical explicit reformulation we implemented. In Section~\ref{sec:numericalresults}, several numerical test cases are carried out to evaluate the ability of the model to capture the jamming effect in collective motions and the transition towards flocking under some specific assumptions. We end the paper with conclusions and outlook in Section~\ref{sec:conclusion}. \reviewerB{Finally, Appendix~\ref{appendix:comput_time} provides details about the computational framework and Appendix~\ref{app:videos} describes the videos contained in the supplementary material.}

\section{Description and analysis of the \reviewerA{model}}
\label{sec:model_analysis}

We first present in this section the derivation of the model we designed to describe the collective cell dynamics and the transitions from jamming to flocking under the influence of different parameters.The construction is guided by the idea of a ``minimal'' set of ingredients that are necessary for the description of the key underlying mechanisms at play in the system. The structure stems from a novel model we proposed and validated in~\cite{nature_phys}, guided by experiments on cellular rings. In the original version, cell dynamics was described by means of a Vicsek-type model in interaction with self-generated surrounding acto-myosin cables. However, as the focus of the present work is on cell-cell and cell-wall interactions, two major differences are implemented: (i) we introduce a soft local attraction-repulsion force derived from a potential to model elastic interactions and (ii) we do not take into account the impact of the supra-cellular cables on the cellular dynamics. Note that from a biological viewpoint, the description is typically focused on the collective movement of cells, but comparable rules could apply in the case of animal or human groups of agents.

More precisely, each cell is represented as a hard-sphere of constant radius $R_0>0$, which means that its deformation is neglected. We consider $N$ cells, moving in a fixed two-dimensional domain $\Omega \subset \R^2$. For each cell, the dynamics are characterized by the position $\mathbf{X}_{k}(t) \in \R^{2}$, the velocity $\mathbf{V}_{k}(t) \in \R^{2}$ and the polarity $\mathbf{P}_{k}(t) \in \mathbb{S}^{1}$, where $\mathbb{S}^{1}$ denotes the set of unit vectors $\mathbb{S}^{1} = \{\mathbf{d} \in \R^{2}, \|\mathbf{d}\| = 1\}$, $\|\cdot\|$ is the Euclidean norm and $k = 1, \ldots, N$.  We denote by $\mathbf{X}(t) = (\mathbf{X}_{k}(t))_{ k = 1,\ldots, N} \in (\R^{2})^{N}$ the vector of all positions, by $\mathbf{V}(t) = (\mathbf{V}_{k}(t))_{ k = 1,\ldots, N} \in (\R^{2})^{N}$ the vector of all velocities and by $\mathbf{P}(t) = (\mathbf{P}_{k}(t))_{ k = 1,\ldots, N} \in (\mathbb{S}^{1})^{N}$ the vector of all polarities, respectively. We refer to Fig.~\ref{fig:schematic} for a schematic of the geometrical configuration.

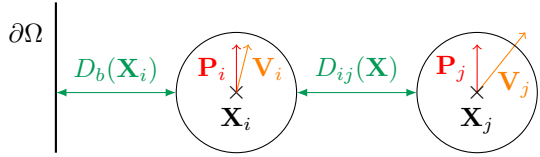
\begin{figure}
\begin{center}
    \begin{tikzpicture}[scale = 0.8]
    \draw (0,0) circle (1cm);
    \node at (0,0) {$\times$}; 
    \node[below] at (0,-0.1) {$\X_i$};  Label for the first circle

    node
    \draw[orange,->] (0,0) -- (0.2,0.8) node[midway, right, orange] {$\V_i$};
    \draw[orange,->] (4,0) -- (4.8,1) node[midway, below, orange] {$\quad \V_j$};
    
    \draw[red,->] (0,0) -- (0,0.8) node[midway, left, red] {$\Po_i$};
    \draw[red,->] (4,0) -- (4,0.8) node[midway, left, red] {$\Po_j$};
    
    \draw (4,0) circle (1cm);
    \node at (4,0) {$\times$}; 
    \node[below] at (4,-0.1) {$\X_j$}; 

    \draw[>=latex,<->,ForestGreen] (1,0) -- (3,0) node[midway,above] {\color{ForestGreen}$D_{ij}(\X)$};

    \draw[thick] (-3,-1) -- (-3,1.5) ;
    \node at (-3.5,1) {$\partial \Omega$};

    \draw[>=latex,<->,ForestGreen] (-3,0) -- (-1,0) node[midway,above] {\color{ForestGreen}$D_b(\X_i)$};
    \end{tikzpicture}
\end{center}
\caption{Schematic presenting the geometrical configuration. $(\X_i, \V_i, \Po_i)$: position, velocity and polarity of the $i$-th cell. $D_{ij}(\X)$: distance between the $i$-th and $j$-th cells. $D_b(\X_i)$: distance between the $i$-th cell and the boundary.}
\label{fig:schematic}
\end{figure}

\paragraph{Positions and velocities dynamics.} 

We start by writing the fundamental laws of dynamics, implying that the positions and velocities time evolution is given by the following equations:
\begin{align}
    &\displaystyle \frac{d \X}{dt} =  \V, \label{eq:pos}\\
     &\V =  \Proj_{\mathcal{C}_\X}(c\Po + \gamma \mathbf{F}(\X)). 
     \label{eq:vel}
\end{align}
Eq.~\eqref{eq:pos} simply translates that the velocity represents the time derivative of the position. In Eq.~\eqref{eq:vel} several ingredients are gathered. We first assume that cell velocity aligns to polarity and that $c > 0$ is the \reviewerA{desired} velocity amplitude. Next, following ~\cite{graner}, the impact of a soft attraction-repulsion force on the velocity is incorporated, by means of a force derived from a potential. Therefore, $\mathbf{F} = (\mathbf{F}_k)_{ k = 1,\ldots, N}$ is described by: 
\begin{equation}
    \mathbf{F}_k (\X) = \displaystyle \sum_{j, \norme{\X_k-\X_j}\leq \Rintar}  \nabla_{\X_k} W(\lVert \X_k - \X_j \rVert)
\end{equation}
for each $k = 1,\ldots, N$, where $W$ is an attraction-repulsion potential and $\Rintar> 0$ denotes the cell attraction-repulsion interaction radius. We here use the following polynomial potential~\cite{graner}: 
\begin{equation}
    W(r)  = -\kappa \, \left(\frac{r^2}{2}-\frac{r^3}{6R_{c}}\right),
\end{equation}
where $R_{c}> 0$ is the radius of a ``comfort zone'' that we introduce in the neighborhood of each cell and $\kappa >0 $ is the rigidity constant associated with the attraction-repulsion force. The parameter $\gamma >0$ represents the inverse friction (or mobility) coefficient that controls the amplitude of the velocity response to forces~\cite{graner}. \reviewerA{We note that, although the model actually depends only on the multiplicative constant $\gamma\kappa$, retaining the two parameters allows the two underlying physical phenomena to be interpreted separately.} Finally, to prevent cell overlapping and to properly characterize cell dynamics in the vicinity of the walls (which are the borders of the domain $\Omega$), we follow the so-called ``hard-repulsion'' approach, introduced in~\cite{collision_maury,MauryVenel2011}. The idea is to project the velocity \reviewerA{$(c\Po + \gamma \mathbf{F}(\X))$} that the particles would have without obstacles on a set of admissible velocities, denoted $\mathcal{C}_\X$, which prevents cells from overlapping and escaping the domain:
\begin{align}
\mathcal{C}_\X = \{ \V \in \mathbb{R}^{2N} | \;\forall i<j, \; D_{i,j}(\X)=0 &\implies \nabla D_{i,j}(\X)\cdot\V \geq 0, \nonumber \\ 
\;\forall i, \; D_{b}(\X_i)=0 &\implies \nabla D_{b}(\X_i)\cdot\V_i \geq 0\}.\label{eq:CX}
\end{align}
We here denote by $D_{i,j}(\X)  = \norme{\X_i-\X_j} - 2R_0$ the distance between the $i$-th and $j$-th
particles of radius $R_0$ and by $D_b(\X_i) =  \inf_{\mathbf{y} \in \partial \Omega}\norme{\mathbf{y}-\X_i} - R_0$ the distance between the $i$-th cell and the boundary (see Fig.~\ref{fig:schematic}). The notation $\Proj_{\mathcal{C}_\X}$ stands for the projection operator on the set of admissible velocities, that incorporates nonlocal contributions from the neighbors on the dynamics of each cell.

\paragraph{Polarities dynamics.}
We consider here that the polarity of each cell is defined by the direction of its lamellipodia and it reflects the anisotropic distribution of its cytoskeleton. Again, several mechanisms are incorporated in the description of the time evolution of the polarity, as follows:
\begin{equation}
    d\Po_k =\Proj_{\Po_k^\perp} \circ \left( \mu\, (\overline{\Po}_k -\Po_k)\, dt + \delta\, \left(\frac{\V_k}{\norme{\V_k}}-\Po_k\right) dt + \sqrt{2D}\, (d\mathbf{B}_t)_k \right),
    \label{eq:pol}
\end{equation}
for each $k= 1, \ldots, N$. First, in the spirit of a Vicsek-like modeling approach, see for instance~\cite{hakim2017collective} for a review of such developments in the context of cell dynamics modeling, we assume that cells align their polarities with their neighbors. We therefore introduce the following local average polarity:
\begin{equation}
    \overline{\Po}_k = \displaystyle \frac{ \displaystyle \sum_{j, \norme{\X_j - \X_k}\leq \Rintpo}\Po_j}{ \norme{\displaystyle \sum_{j, \norme{\X_j - \X_k}\leq \Rintpo }\Po_j}},
\end{equation}
\reviewerA{where $\Rintpo > 0$ denotes the radius of the polarity interaction zone for each cell.} We assume that the polarity of the $k$-th cell relaxes to the \reviewerA{local average polarity $\overline{\Po}_k$} of the neighboring particles with an angular rate $\mu > 0$.  Next, we incorporate the relaxation of the polarity to the velocity unit direction $\V_k /\norme{\V_k}$, at angular rate $\delta > 0$. The last contribution in Eq.~\eqref{eq:pol} corresponds to a Gaussian white random noise with an angular diffusion coefficient $D > 0$. Finally, to ensure that the polarity of each cell remains of norm 1 during time, a projection step of the possible infinitesimal variation is performed onto the orthogonal space to the vector $\Po_k$, by means of the projection operator $\Proj_{\Po_k^\perp}$. The symbol $\circ$ corresponds to the fact that Eq.~\eqref{eq:pol} is to be understood in the Stratonovich sense \cite{evans2012introduction}.

\paragraph{Well-posedness.} We next study the well-posedness of the problem described by Eqs.~\eqref{eq:pos}-\eqref{eq:vel}-\eqref{eq:pol}, without the contribution of the stochastic part ($D=0$), following the method proposed in \cite{MauryVenel2011}. The novelty lies in the fact that in the present work we account for the polarity dynamics and consequently, the resulting system can be seen as an extension of the previously studied problem, where only position and velocity equations were coupled.

We first rewrite the polarity equation using polar coordinates. Introducing the polarity angles $\theta_k \in [0,2\pi)$, such that $\Po_k = (\cos(\theta_k), \sin(\theta_k))^T$ (where $(\cdot)^T$ denotes the transposition operator), Eq.~\eqref{eq:pol} rewrites as follows:
\begin{equation*}
    \frac{d \theta_k}{dt} \left( \begin{array}{c}
             - \sin(\theta_k) \\
             \cos(\theta_k)
        \end{array} \right)  
         = \left[\left( \begin{array}{c}
             - \sin(\theta_k) \\
             \cos(\theta_k)
        \end{array} \right) \cdot \left(\mu \overline{\Po}_k +\delta \frac{\V_k}{\norme{\V_k}}\right)\right] \left( \begin{array}{c}
             - \sin(\theta_k) \\
             \cos(\theta_k)
        \end{array} \right).
 \end{equation*}
Then, denoting by $\bar\theta_k$ and $\psi_k$ the angles of the vectors $\overline{\Po}_k$ and $\V_k/\norme{\V_k}$, after taking the scalar product with $\Po_k^\perp = ( -\sin(\theta_k), \cos(\theta_k))^T$, the previous equation simply writes: 
\begin{equation*}
 \frac{d \theta_k}{dt} 
          =  \mu \sin(\bar\theta_k - \theta_k) + \delta \sin(\psi_k - \theta_k).
\end{equation*}
Therefore, the whole system can be recast as:
\begin{equation}
\label{eq:proj_complet}
\frac{d}{dt} ( \X, \theta) = \Proj_{\mathcal{C}_\X \times \mathbb{R}^N} \U( \X, \theta),
\end{equation}
where
\begin{equation}
\label{eq:U}
\U(\X,\theta)_j = \begin{cases}
 c\, (\cos{\theta_j}, \sin{\theta_j})^T + \gamma \mathbf{F}_j(\X), & \text{if } 1 \leq j \leq N,\\
    \mu\, \sin(\bar\theta_{j-N} - \theta_{j-N}) + \delta \, \sin(\psi_{j-N} - \theta_{j-N}), & \text{if } N+1 \leq j \leq 2N.
    \end{cases}
\end{equation}
We note that $\psi_k$ is actually a function of $\X$ and $\theta$. 

To study this equation, we transform it into a weak differential inclusion. Denoting by $\mathcal{N}_\X$ the polar cone\footnote{A polar cone $C^\circ$ of a convex cone $C$ in $\mathbb{R}^n$ is defined as $C^\circ = \{ \mathbf{y} \in \mathbb{R}^n \mid \mathbf{x} \cdot \mathbf{y} \leq 0 \ \text{for all} \ \mathbf{x} \in C\}$.} of the non-empty convex set $C_{\X}$, it can be easily checked that the polar cone of $C_{\X} \times \mathbb{R}^N$ is $\mathcal{N}_{\X} \times \{0\}$.  Therefore, we have the following identity:
\begin{equation*}
    \Proj_{\mathcal{C}_\X \times \mathbb{R}^N} + \Proj_{\mathcal{N}_\X \times \{0\}} = \text{Id},
\end{equation*}
and thus Eq.~\eqref{eq:proj_complet} can also be written as:
\begin{equation*}
    \frac{d}{dt} ( \X, \theta) = \mathbf{U}(\X, \theta) -\Proj_{\mathcal{N}_\X \times \{0\}} \mathbf{U}(\X, \theta).
\end{equation*}
In this formulation, the last term can be interpreted as the contact forces exerted on the cells. In particular, the solution to the differential equation satisfies the following differential inclusion:
\begin{equation}
        \displaystyle \frac{d(\X,\theta)}{dt}  + \mathcal{N}_\X \times \{0\}  \ni  \mathbf{U}(\X, \theta).
    \label{rafle}
\end{equation}
 The existence and uniqueness of a solution to this equation \reviewerA{and equivalently to} Eq. \eqref{eq:proj_complet} are stated in the following proposition, \reviewerA{which is an application of Theorem 2.10 in \cite{MauryVenel2011}. We also refer to \cite{these_venel} for additional details.}

\begin{prop}
\label{prop:exist}
Let \( Q = \{ \X \in \mathbb{R}^{2N} \mid \forall i < j, \, D_{ij}(\X) \geq 0 \}\) be the set of admissible configurations and assume that \( \mathbf{U} \) is Lipschitz and bounded. Then, for any initial data \( (\X_0, \theta_0) \in Q \times \mathbb{R}^N \) and any time \( T > 0 \), there exists a unique absolutely continuous solution\footnote{A function $f:[0,T]\to\R^N$ is said absolutely continuous if each component is in $L^1(0,T)$ and has a weak derivative in $L^1(0,T)$. The set of absolutely continuous solution of the interval $(0,T)$ is also denoted $(W^{1,1}(0,T))^N$.} on the interval $(0,T)$ to the system
\begin{equation}
\begin{cases}
\displaystyle \frac{d(\X,\theta)}{dt}  + \mathcal{N}_\X \times \{0\}  \ni \mathbf{U}(\X, \theta), \\
(\X, \theta)(0) = (\X_0, \theta_0).
\end{cases}
\end{equation}
\reviewerA{Moreover, this is also the unique absolutely continuous solution to Eq. \eqref{eq:proj_complet}.}
\end{prop}

\begin{proof} As stated in the proof of Theorem 2.30 (p.~40) in \cite{these_venel}, the existence and result follows from the fact that   $\mathcal{N}_\X \times \{0\}$ coincides with the normal proximal cone of $Q\times \R^N$ at point $(\X,\theta)$ and that $Q\times\R^N$ is uniformly prox-regular. \reviewerA{Moreover, these properties also imply the equivalence with the initial problem as stated in \cite{bernicot2010differential} (prop 3.3).} \reviewerAA{Let us check that these two statements are true.}

From \cite{these_venel} (Prop. 3.9 p. 48) , the polar cone $\mathcal{N}_{\X}$ actually coincides with the proximal cone of $Q$ at $\X$, which is denoted $N(Q, \X)$ and is defined by: 
\begin{equation}
     N(Q, \X) = \{\V \in \R^{2N}\ |\ \exists \alpha >0, \X \in \Proj_Q(\X + \alpha \V) \}.
\end{equation}
where $\Proj_Q(\mathbf{Y})$ denotes the set of all elements of $Q$ that realize the minimal distance between $\mathbf{Y}$ and $Q$. Adding the angle variable results in the following identities:
\begin{align*}
    N(Q \times \mathbb{R}^N, (\X, \theta)) &= \{(\V, \psi) \in \R^{3N}\ |\ \exists \alpha >0, (\X, \theta) \in \Proj_{Q \times \mathbb{R}^N}(\X+\alpha \V,\theta + \alpha \psi) \}\\
    &= \{(\V, 0) \in \R^{3N}\ |\ \exists \alpha >0, \X \in \Proj_{Q}(\X+\alpha \V) \}\\
    &=  N(Q, \X) \times \{0\},
\end{align*}
since, for any vectors $\mathbf{a}$, $\mathbf{b}$ and any sets $A$, $B$, we have $\Proj_{A\times B}(\mathbf{a},\textbf{b}) = \Proj_A(\mathbf{a})\times \Proj_B(\mathbf{b})$ and $\theta \in \Proj_{\mathbb{R}^N}(\theta + \alpha \psi) = \{\theta + \alpha \psi\}$ results into $\psi = 0$ as $\alpha$ is strictly positive. Consequently, the polar cone $\mathcal{N}_{\X}\times\{0\}$ coincides with the proximal cone of $(Q\times \R^N)$ at $(\X, \theta)$.

From \cite{these_venel} (Prop. 3.12 p. 51), the set $Q$ is prox-regular with a positive constant $\eta > 0$, meaning that 
\begin{equation}
    \mathcal{B}(\X + \eta \V /\| \V \|,\eta) \cap Q = \emptyset,\quad \forall \X \in \partial Q, \forall \V \in N(Q,\X)\backslash \{0\},
\end{equation}
where $\mathcal{B}(\mathbf{x}, r)$ denotes the ball with center $\mathbf{x}$ and radius $r>0$. This yields the uniform prox-regularity property for the set $Q\times\R^N$
\begin{equation}
   \mathcal{B}((\X + \eta \V /\| \V \|,\theta),\eta) \cap (Q\times\R^N) = \emptyset,\quad \forall (\X,\theta) \in \partial (Q \times \R^N), \forall (\V,0) \in N(Q \times \mathbb{R}^N, (\X, \theta)) \backslash \{0\}, 
\end{equation}
since $\mathcal{B}((\X + \eta \V /\| \V \|,\theta),\eta) \subset \mathcal{B}(\X + \eta \V /\| \V \|,\eta)\times \R^N$.
\end{proof}

\begin{remark}
The function $\U$, defined in Eq. \eqref{eq:U}, is bounded but not Lipschitz as both the attraction-repulsion force and the average polarity depend on a discontinuous cut-off function and a projection on the circle respectively. However, regularized interactions can be added to smooth out these functions and therefore to apply Prop.~\ref{prop:exist} in order to conclude on \reviewerA{the well-posedness of the regularized version} of the coupled system ~\eqref{eq:pos}-\eqref{eq:vel}-\eqref{eq:pol}. 
\end{remark}

\section{Numerical method}
\label{sec:numericalmethod}
 
In this section, we present the numerical method implemented to solve Eqs.~\eqref{eq:pos}-\eqref{eq:vel}-\eqref{eq:pol}. Let $\Delta t >0$ be the time step and denote by $(\X^n_k)$, $(\V^n_k)$ and $(\Po^n_k)$ the approximate positions, velocities and polarities at time $t^n=n \Delta t$, $n\in \reviewerA{\mathbb{N}}$, respectively. The method combines a semi-implicit approach introduced in \cite{numerical_polarity} for the polarity dynamics~\eqref{eq:pol} and the method proposed in \cite{MauryVenel2011} for the discretization of the equations coupling positions and velocities~\eqref{eq:pos}-\eqref{eq:vel}.

\paragraph{Discretization of the positions and velocities dynamics.} We consider the following semi-implicit time-stepping algorithm:
\begin{align} 
\X^{n+1} &=   \X^n + \Delta t\,\V^{n+1},\label{eq:pos_update}\\
 \V^{n+1} &= \Proj_{\mathcal{C}_{\X^n}^{\Delta t}}(c\Po^{n+1} + \gamma \mathbf{F}(\X^n)),
\label{eq:XV_disc}
\end{align}
where $\mathcal{C}_{\X^n}^{\Delta t}$ is a first-order approximation of the admissible velocity set $\mathcal{C}_{\X^n}$ described as follows:
\begin{align*}
\mathcal{C}_{\X^n}^{\Delta t} = \{ \V \in (\mathbb{R}^2)^N \; | \; \forall i<j, &\; D_{i,j}(\X^n) + \Delta t\, \nabla D_{i,j}(\X^n)\cdot\V \geq 0,  \\
 \forall i, \; &\ D_{b}(\X^n_i) + \Delta t\, \nabla D_{b}(\X_i^n)\cdot\V_i \geq 0\}.
\end{align*}
Introducing the vectors  $\D =(D_{i,j}(\X^n))_{i<j} \in \mathbb{R}^{\frac{N(N-1)}{2}}$, $\D_{b} =(D_{b}(\X_i^n))_{i} \in \mathbb{R}^{N}$, the matrices $B$  of size $(N(N-1)/2, 2N)$ and $B_b$  of size $(N, 2N)$ such that 
\begin{equation*} B \V = ( -\Delta t\ \nabla D_{i,j}(\X^n)\cdot\V)_{i<j} \in \mathbb{R}^{N(N-1)/2},\quad 
B_b \V = ( -\Delta t\ \nabla D_{b}(\X_i^n)\cdot\V)_{i}\in \mathbb{R}^{N}  \quad \forall\ \V \in (\mathbb{R}^2)^N,
\end{equation*} and denoting by 
\begin{equation*}\widetilde{\D} = \begin{bmatrix} \D\\ \D_b \end{bmatrix} \in \R^{\frac{N(N-1)}{2}+N},\quad 
\widetilde{B} = \begin{bmatrix} B\\ B_b \end{bmatrix} \in {\cal M}_{\frac{N(N-1)}{2}+N,2N}\ ,
\end{equation*}
the admissible velocity set can be rewritten as:
\begin{align}
\label{eq:constr}
\mathcal{C}_{\X^n}^{\Delta t} =  \{ \V \in (\mathbb{R}^2)^N \; |  & \; \widetilde{B} \V - \widetilde {\D} \leq 0 \},     
\end{align}
where the constraint should be understood component-wise.
Following the approach introduced in~\cite{collision_maury}, we reformulate the projection step in Eq.~\eqref{eq:XV_disc} as an optimization problem:
\begin{equation}
    \V^{n+1} = \underset{\V \in \mathcal{C}_{\X^n}^{\Delta t}}{\text{argmin}} \frac{1}{2}\norme{\V - c \Po^{n+1} - \gamma \mathbf{F}(\X^n)}^2,
    \end{equation}
and introduce the associated Lagrangian functional, namely:
\begin{equation}
    \mathcal{L}(\V, \blambda) = \frac{1}{2} \norme{\V - c \Po^{n+1} - \gamma \mathbf{F}(\X^n)}^2 + \blambda \cdot \left( \widetilde{B}\V-\widetilde{\D}\right),
\end{equation}
where $\blambda \in \R^{\frac{N(N-1)}{2}+N}$ is the Lagrange multiplier associated with the constraints defined in Eq.~\eqref{eq:constr}. At each time step, we solve the minimization problem using the Uzawa algorithm~\cite{Ciarlet2007}, which consists in approximating a solution $\left( \V^{n+1},\blambda \right)$ by constructing the sequences $(\V^{(j)})_j$ and $(\blambda^{(j)})_j$ as follows:
\begin{align}
1.\quad &\V^{(0)} = \V ^n, \blambda^{(0)} = \mathbf{0},\label{eq:uzawa1}\\
2.\quad &\V^{(j+1)} = \displaystyle \min_{\V \in \mathbb{R}^{2N}} \mathcal{L}(\V, \blambda^{(j)}) = c\Po^{n+1} + \gamma \mathbf{F}(\X^n) - \widetilde{B}^T \blambda^{(j)},\label{eq:uzawa2} \\
3.\quad &\blambda^{(j+1)} = \max \left ( \mathbf{0}, \blambda^{(j)} + h\left ( \widetilde{B}\V^{(j)}-\widetilde{\D}\right ) \right ),\label{eq:uzawa3}
\end{align}
 where $h >0$ is the gradient-descent step of the method and max is the component-wise maximum function. Let us mention that an alternative method to solve \eqref{eq:pos}-\eqref{eq:vel} would be the Arrow-Hurwitz algorithm, which is not based on the linearization of the constraints \cite{Degond_2017,DegondSIADS}.

\begin{remark} [Convergence analysis of the Uzawa algorithm]
According to \cite[Th.~9.4-1]{Ciarlet2007}, the Uzawa algorithm converges to the solution of the minimization problem if the gradient descent step $h$ satisfies $0 < h < \reviewerA{h_{\max} =}\ 2 \alpha/|\!|\widetilde{B}|\!|^2$. Moreover, a rough lower bound of $h _{\max}$ was derived in~\cite{collision_maury}:
$$h_{\max} \geq \frac{1}{2 N_{\max} \Delta t ^2 \sqrt{d}},$$ where $N_{\max}$ is the maximal number of contacts, which in dimension $d=2$ and in the case of constant radius for all particles is equal to $6$. We therefore take $h$ equal to $\left ( 12 \sqrt{2} \Delta t ^2 \right ) ^{-1}$ to ensure convergence of the algorithm under the stopping criterion on the relative distance between two iterative values of the velocity field below a threshold value fixed at $10^{-2}$.
\end{remark}

\paragraph{Discretization of the polarities dynamics.} 
Recalling that the norm of the polarity vector is equal to 1 during time, a  suitable discretization method that preserves this property at the discrete level has to be implemented. We therefore consider the following discretization of the polarity equation \eqref{eq:pol}, initially proposed in~\cite{numerical_polarity}:
\begin{equation*}
    \Po_k^{n+1} = \Po_k^n + \left(\text{Id} - \Po_k^{n+1/2}\otimes \Po_k^{n+1/2}\right) \left( \mu\,\Delta t\, \left(\overline{\Po}_k^n-\Po_k^n\right)  + \delta\,\Delta t \, \left(\frac{\V_k^n}{\norme{\V_k^n}} - \Po_k^n\right)+ \sqrt{2D \Delta t}\, \boldsymbol{\xi}_k^n \right),
\end{equation*}
where $\Po_k^{n+1/2} = (\Po_k^n + \Po_k^{n+1})/2$, the symbol $\otimes$
denotes the tensor product of two vectors $\mathbf{a}$, $\mathbf{b}$
defined by $\mathbf{a} \otimes \mathbf{b} = \mathbf{a} \mathbf{b}^T$
and $\boldsymbol{\xi}_k^n$ is a random number following a standard Normal distribution. In terms of polarity angles $\theta^n_k \in [0,2\pi)$, such that $\Po ^n _k = \left(\cos \left(\theta^n_k\right), \sin\left(\theta^n_k\right)\right)^T$, the semi-implicit scheme can be expressed under the following explicit formulation:
\begin{equation}
    \theta^{n+1}_k  =  \theta^n_k + 2 \left( \hat{Q}^n_k -\theta^n_k\right) + \sqrt{2D \Delta t}\, \boldsymbol{\xi}^n_k, 
    \label{eq:theta_disc}
\end{equation}
with $\Q_k^n$ defined by:
\begin{equation*}
    \Q^n_k = \Po_k^n +\frac{ \Delta t }{2} \left(\mu\,\left(\overline{\Po}_k^n -\Po_k^n \right) + \delta \,\left(\frac{\V_k^n}{\norme{\V_k^n}} - \Po_k^n\right) \right).
\end{equation*}
and $\hat{\Q}^n_k$ is the polar angle of the vector $\Q^n_k$.\\

To sum up, the numerical method consists at each time step in a three-stage approach: first, compute the polarity vector $\Po^{n+1}$ of unit norm knowing $\Po^n$ and $\V^n$ using \eqref{eq:theta_disc}, then update the velocity vector $\V^{n+1}$ by means of the Uzawa algorithm \eqref{eq:uzawa1}-\eqref{eq:uzawa2}-\eqref{eq:uzawa3} and on the basis of $\Po^{n+1}$ and $\mathbf{F}(\X^n)$, and finally compute the new position vector $\X^{n+1}$ using $\X^n$ and $\V^{n+1}$ using \eqref{eq:pos_update}.

\section{Numerical experiments and discussion}
\label{sec:numericalresults}

In this section, different test-cases are performed to show that the model is able to capture flocking phase transitions and jamming effects in the context of cell collective dynamics.  The reference parameters of the model are listed in Table \ref{tab:parameter}. They are all extracted from \cite{nature_phys}, except the attraction-repulsion interaction radius $\Rintar$, the cell comfort radius $R_c$ and the rigidity constant $\kappa$, which are specific to Model \eqref{eq:pos}-\eqref{eq:vel}-\eqref{eq:pol}. Note that, in this set of parameters, the comfort distance $R_c$ has been taken equal to half the attraction-repulsion radius $\Rintar$. Consequently, only the soft repulsion part of the force is activated. \reviewerA{Note also that the relaxation parameters, $\mu$ and $\delta$, are taken equal as they correspond to the same underlying biological process, that is the reorganization of the cytoskeleton of the cells. We will still conduct a detailed study of their relative impact on the dynamics in Section~4.2.}

\reviewerB{The numerical method is implemented in an in-house Python computational framework, whose performances in terms of computing time are presented in Appendix~\ref{appendix:comput_time}.}

\begin{table}[H]
    \rowcolors{1}{gray!05}{white}
    \centering
    \begin{tabular}{llll}
    \toprule
    \textbf{Cell radius} &  $R_0$ & $7.5$ &$\microm$\\
    \textbf{Cell comfort radius} &  $R_{c}$ & $9.5$ &$\microm$\\
    \textbf{Cell attraction-repulsion interaction radius}  &  $\Rintar$ & $19$ &$\microm$\\
    \textbf{Cell polarity interaction radius} &  $\Rintpo$ & $60$ &$\microm$\\
    \textbf{Cell speed} & $c$ & $21.6$ &$ \microm \, \h^{-1}$ \\
    \textbf{Angular diffusion} & $D$ & $0.96$ &$\rad^2\, \h^{-1}$ \\
    \textbf{Relaxation parameter: polarity to mean polarity} & $\mu$ & $6.2$ &$\rad\, \h^{-1}$ \\
    \textbf{Relaxation parameter: polarity to velocity} & $\delta$ & $6.2$ &$\rad\, \h^{-1}$ \\
    \textbf{Rigidity constant} & $\kappa$ & $10^4$ &$\pN\, \microm^{-1}$ \\
    \textbf{Inverse friction coefficient} & $\gamma$ & $10^{-5}$ &$\pN^{-1} \h^{-1} \microm$ \\
    \bottomrule
    \end{tabular}
    \caption{Parameters of the model.}
    \label{tab:parameter}
\end{table}

\subsection{Order-disorder phase transition in a periodic domain}

First, we study the influence of the angular diffusion parameter $D$ on the global alignment of the polarities. In order to assess the impact of this factor in quantitative terms, we introduce the polarity order parameter $\phi$ defined as:
\begin{equation*}
    \phi(t) = \norme{\frac{1}{N}\sum^N_{k=1}\Po_k(t)}.
\end{equation*}
This quantity equals $1$ when all the polarities are aligned and is close to $0$ when they are randomly distributed.

\reviewerA{For this study, we consider periodic boundary conditions: the effect of boundaries is thus ignored.} Specifically, cells move in a square domain of size $L = 200\ \microm$ with periodic boundary conditions and the constraints due to the boundary conditions are removed from the set of the admissible velocities $\mathcal{C}_\X$ defined in Eq.~\eqref{eq:CX}. The cell dynamics are computed up to time $T= 20\,\h$  and the order parameter is averaged over the last $ T/8 = 2.5\, \h$ of the simulations:
\begin{equation}
     \frac{1}{\sharp \{t^n \in [7T/8,T]\}}\sum_{t^n \in [7T/8,T]} \phi(t^n),
    \label{eq:phip}
\end{equation}
where $\sharp A$ denotes the cardinal of a set $A$.

In Fig.~\ref{fig:phi_p}, we show the variation of the order parameter when the diffusion parameter $D$ is increased for different cell densities  \reviewerA{and different time steps}. In the left panel, simulations are performed with $N=160$ cells and in the right panel with $N=190$ cells, resulting in densities $\rho = 0.707$ and $\rho = 0.839$, respectively. Note that the maximal packing density, $\pi/(2\sqrt{3}) \approx 0.907$, is obtained for the hexagonal lattice configuration. For each diffusion parameter $D$, we have performed $20$ different numerical simulations, in order to properly account for the stochastic aspects of the dynamics. Without diffusion ($D= 0$), the order parameter equals $1$ and polarities are fully aligned, as expected in a case without noise. Then, as the diffusion parameter increases, the order parameter decreases up to reaching low values of about $0.1$, revealing the appearance of a phase transition between order and disorder, similar to results reported for a standard Vicsek-like model in \cite{degond2013macroscopic}. As already identified in \cite{vicsek1995novel}, the smoothness of the transition is due to the relatively low number of cells . We observe that the transition takes place for smaller diffusion parameters when density is increased: contact forces can be seen as an additional source of randomness. We note that the variability of the order parameter is larger for intermediate diffusion values, for instance when $D$ is between $5$ and $9\,\rad^2\h^{-1}$, see Fig.~\ref{fig:phi_p1}.  

\reviewerA{In addition, we report in Fig.~\ref{fig:phi_p} the results obtained for different time steps. We observe the convergence of the phase transition curves as $\Delta t$ decreases, with a slightly more significant variability in the intermediate diffusion coefficient range, where randomness still impacts the results. Moreover, in Fig.~\ref{fig:phi_p2}, results are only displayed for $\Delta t = 0.01\,\h$ and $0.001\,\h$, since the value $\Delta t = 0.1\,\h$ leads to numerical errors. Indeed, the time step has to be sufficiently small to capture the diffusion of the polarity angle. In the following, we therefore choose the time step equal to $10^{-2}\,\h$, which leads to changes in polarity angle of standard deviation of $\sqrt{2D \Delta t} \approx 0.13$ rad at each iteration when $D=0.96\ \rad^2 \,\h^{-1} $.}

 \begin{figure}[H]
    \centering
    \begin{subfigure}{0.49\textwidth}
    \includegraphics[width=\textwidth]{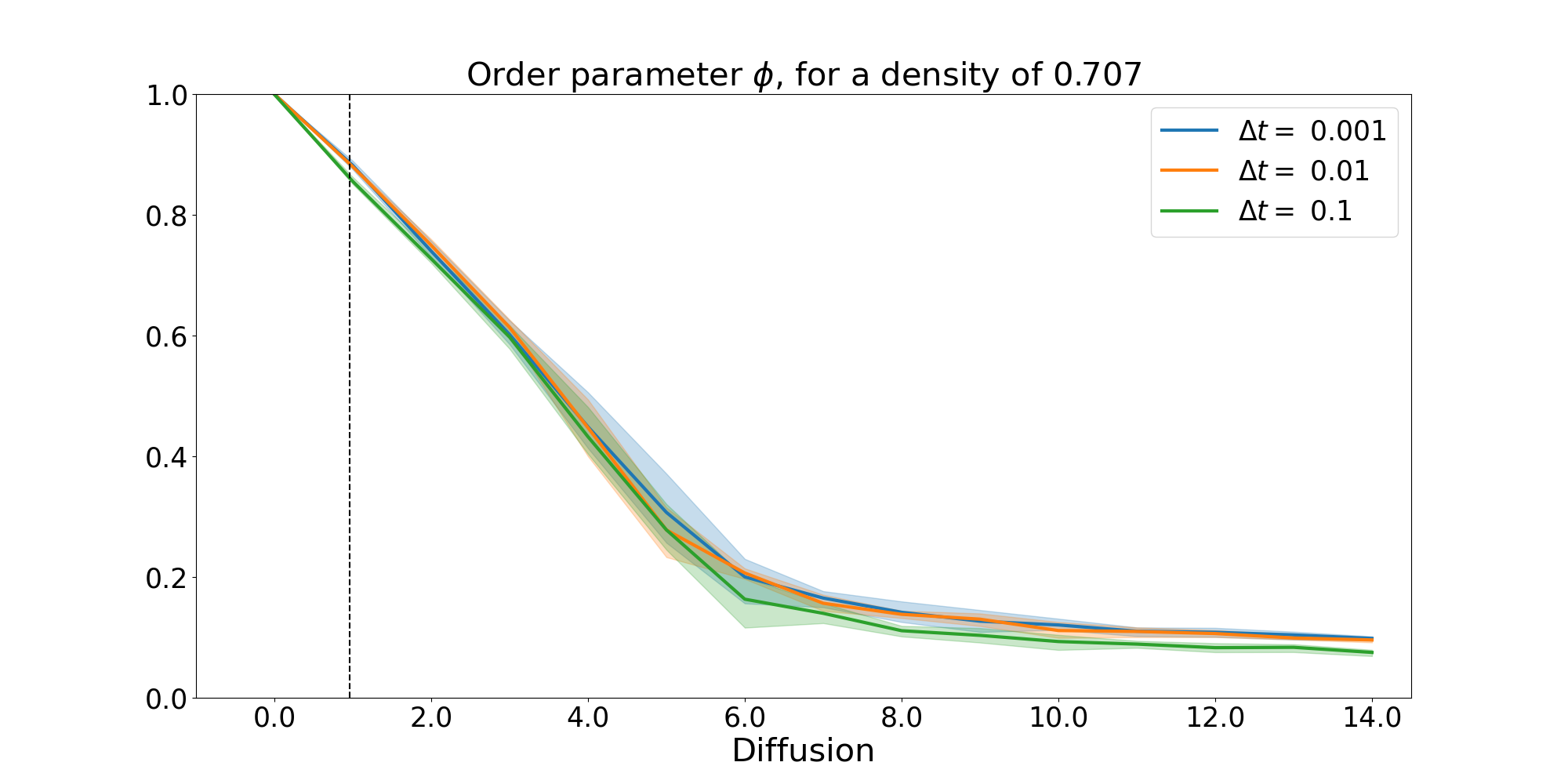}
    \caption{$\rho = 0.707$}
    \label{fig:phi_p1}
    \end{subfigure}
    \begin{subfigure}{0.49\textwidth}
    \includegraphics[width=\textwidth]{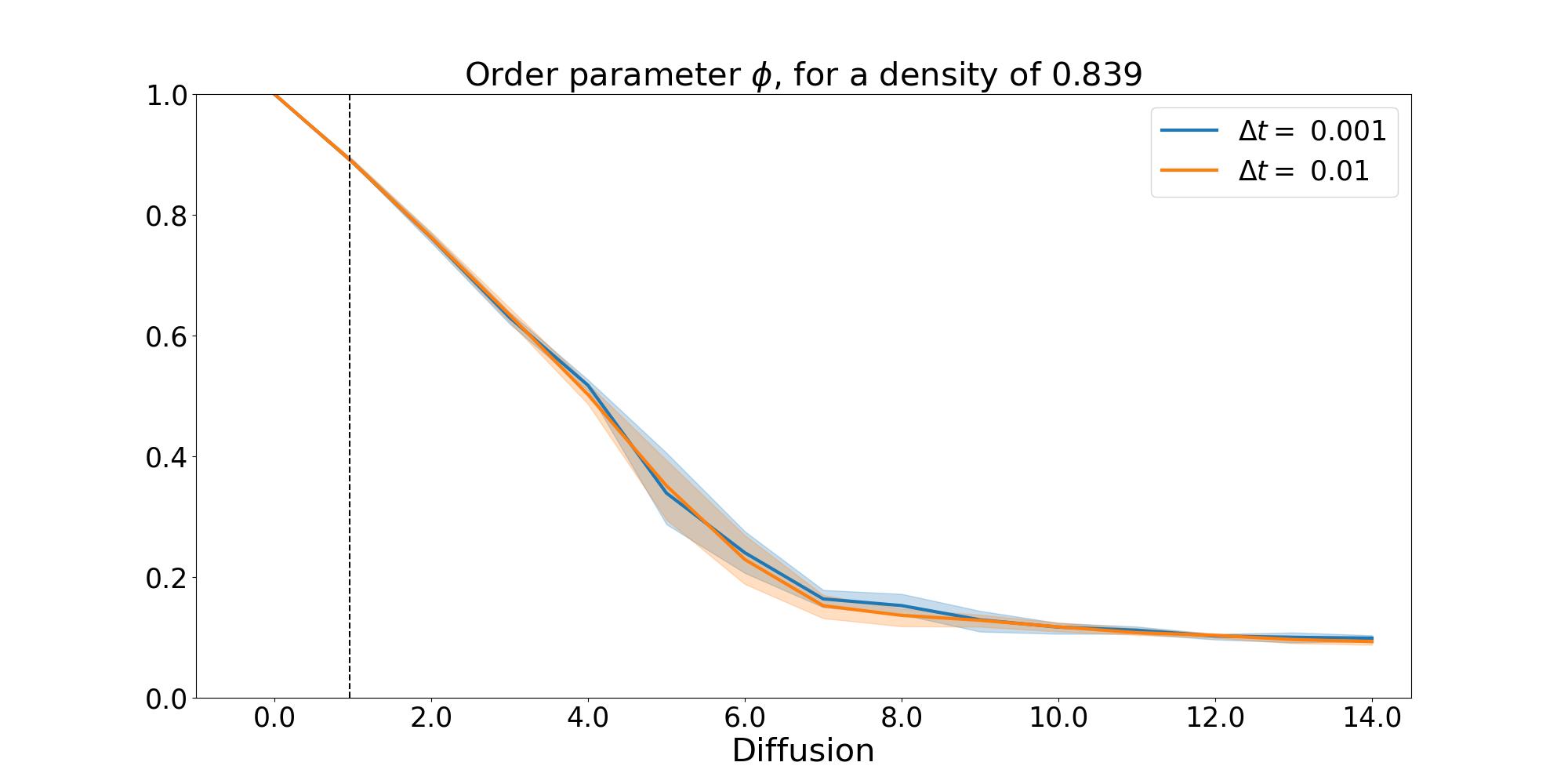}
    \caption{$\rho = 0.839$}
    \label{fig:phi_p2}
    \end{subfigure}
    \caption{(Periodic boundary condition) Order parameter as a function of the diffusion for two different density values, $\rho = 0.707$ (left panel) and $\rho = 0.839$ (right panel). \reviewerA{Mean (solid lines) and
    interquartile interval over $20$ simulations with $T = 20\,h$ and different time steps.} The vertical line at $D =0.96\ \rad^2 \,\h^{-1}$ corresponds to the reference diffusion coefficient used in the other test-cases. 
 }
    \label{fig:phi_p}
\end{figure}

\subsection{Rotation dynamics in a disk} 

In this part we investigate the role of the boundary conditions and the ability of the model to reproduce either a jammed configuration, or a rotational collective movement, depending on the ingredients incorporated in the polarity dynamics. To this end, we consider $N=160$ cells in a disk domain of radius $R = 200/\sqrt{\pi} \,\microm\approx 113\,\microm$. When the velocity feedback on the polarity is deactivated ($\delta = 0$), cells are blocked as depicted in Fig.~\ref{fig:delta0_mu6.2} and the corresponding \Mov{1}. Instead, when the feedback is activated, cells have a collective rotational motion as observed in Fig.~\ref{fig:delta6.2_mu6.2} and \Mov{3}. In order to derive a more accurate assessment of the collective movement, we introduce the three following quantities: the normalized global mean speed, defined as:
\begin{equation}
    \bar{v}(t) = \frac{1}{c}\frac{1}{N}\sum_{k=1}^N \left\|\V_k(t) \right\|,
    \label{eq:barv}
\end{equation}
\reviewerB{where we recall that $c$ is the desired speed of each cell} and the polarity and \reviewerAA{velocity} rotation order parameters:
\begin{equation}
    \phirot(t) = \frac{1}{N}\sum^N_{k=1} \Po_k(t) \cdot \e_k(t),\quad \reviewerAA{\phirot^v(t) = \frac{1}{N}\sum^N_{k=1} \frac{\V_k(t)}{\left\|\V_k(t) \right\|} \cdot \e_k(t),}
    \label{eq:phirot}
\end{equation}
where $\e_{k} = (\X_k - \X_c)^\perp/\| \X_k - \X_c\|$ is the unit tangential vector with respect to the disk center $\X_c$. If $\phirot$ \reviewerAA{(resp. $\phirot^v$)} is close to $1$, the polarity \reviewerAA{(resp. velocity)} vector field is close to a vortex configuration. 

We display in Fig.~\ref{fig:simu} the temporal evolution of the order parameter $\phi$, the normalized mean speed $\bar{v}$ and the rotation order parameters $\phirot$ \reviewerAA{and $\phirot^v$} over the simulations. When the velocity feedback is deactivated ($\delta =0$), see  Fig.~\ref{fig:delta0_mu6.2}, we observe a decrease in mean speed. The polarity rotation order parameter is low and the order parameter reaches values close to $1$\reviewerAA{, while the velocity rotation order parameter oscillates a lot.} These results indicate that cells are less and less moving, do not exhibit a rotational motion but polarities are strongly aligned. On the contrary,  when the velocity feedback is activated ($\delta > 0$), see Fig.~\ref{fig:delta6.2_mu6.2}, the mean speed reaches a value close to $1$, \reviewerAA{the two rotation order parameters now have similar values and still oscillate}, but are closer to $1$. These two indicators confirm the collective rotational motion of the cells. Moreover, the polarity order parameter is below $0.3$, which indicates that polarities are not aligned, which is compatible with a rotational movement. To sum up, these results suggest that when an effective boundary is added, the relaxation of the polarities in the velocity directions plays a crucial role in triggering collective motion. 

To further analyze how the parameters $\delta$ and $\mu$ impact the dynamics,  we gradually vary the parameters $\delta$ and $\mu$ between $0 \; \rad\, \h^{-1}$ and $6.2 \; \rad\, \h^{-1}$ and display the obtained phase diagram in Fig.~\ref{fig:phase_diagram}. Each point has been obtained by averaging the results over $20$ simulations.  In \reviewerAA{Figs.~\ref{fig:phase_diagram_phirot} and ~\ref{fig:phase_diagram_phirot_v}},  we observe that the rotation motion is mostly triggered by the polarity to velocity relaxation parameter $\delta$: indeed, the two rotation order parameters increase sharply as $\delta$ increases. This also induces a sharp increase in the mean speed (Fig.~\ref{fig:phase_diagram_meanspeed}) and a sharp decrease of the order parameter (Fig.~\ref{fig:phase_diagram_phip}). In contrast, the increase is more gradual as the polarity alignment parameter $\mu$ increases. However, the parameter $\mu$ still has an overall effect on the dynamics. To study this point, we compare two simulations with $\delta = 6.2 \; \rad\, \h^{-1}$ and two different values of $\mu$, namely $\mu = 0 \; \rad\, \h^{-1}$ and $\mu = 6.2 \; \rad\, \h^{-1}$.   Representative snapshots and the evolution of quantitative indexes over time are displayed in Figs~\ref{fig:delta6.2_mu0} and \ref{fig:delta6.2_mu6.2}: both simulations present a rotation motion but the absence of polarity alignments leads to more variation in amplitude for the rotation order parameter and the mean speed. For a more in-depth analysis, we also computed the polarity rotation order parameter locally in subdomains (see Fig.~\ref{fig:simu_stat_region}). We observe a high polarity rotation order parameter in the outer ring (touching the edge of the disk) in both cases, but higher for a non-vanishing $\mu$, as expected. In the inner ring, the difference is more pronounced: fluctuations are significantly larger for $\mu =0 \; \rad\, \h^{-1} $. Therefore, polarity alignments help maintain a full, solid rotation.

\begin{figure}
    
    \hfill\begin{subfigure}{0.48\textwidth}
    \centering
    \includegraphics[width= 0.45\textwidth]{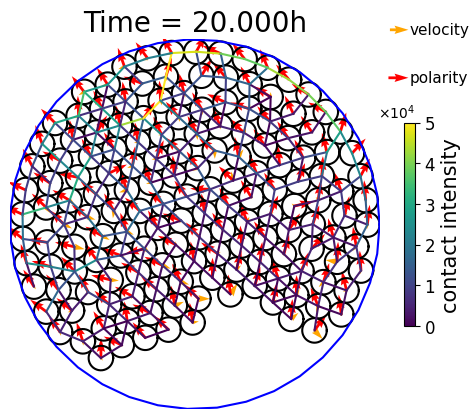}
    \includegraphics[width= 0.45\textwidth]{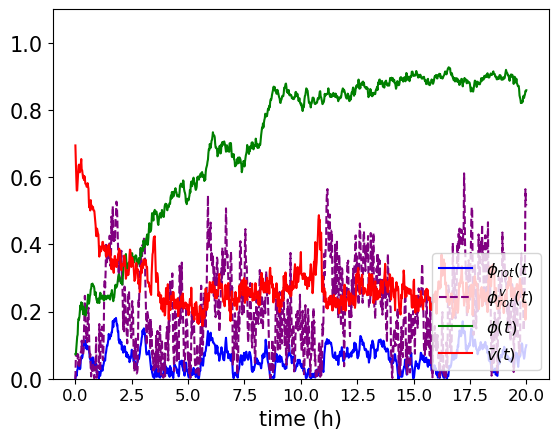}
    \subcaption{$\delta = 0 \; \rad\, \h^{-1}$, $\mu = 6.2 \; \rad\, \h^{-1}$}
    \label{fig:delta0_mu6.2}
    \end{subfigure}

    \begin{subfigure}{0.48\textwidth}
    \centering
    \includegraphics[width= 0.45\textwidth]{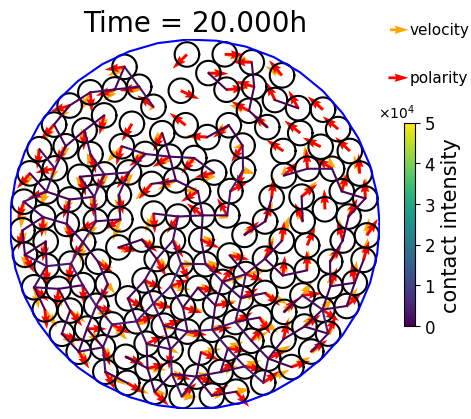}
    \includegraphics[width= 0.45\textwidth]{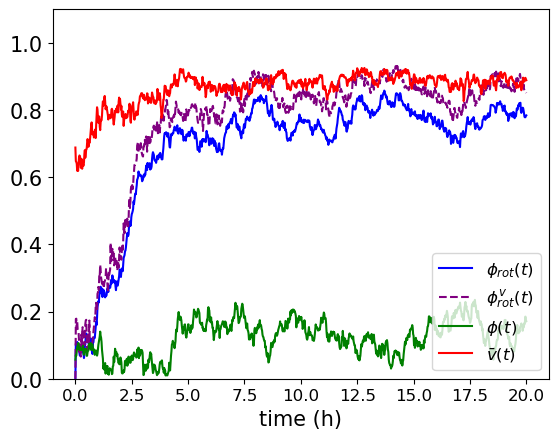}
    \subcaption{$\delta = 6.2 \; \rad\, \h^{-1}$, $\mu = 0 \; \rad\, \h^{-1}$}
    \label{fig:delta6.2_mu0}
    \end{subfigure}
    \hfill\begin{subfigure}{0.48\textwidth}
    \centering
    \includegraphics[width= 0.45\textwidth]{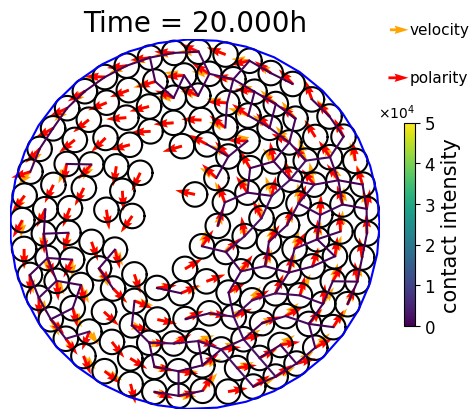}
    \includegraphics[width= 0.45\textwidth]{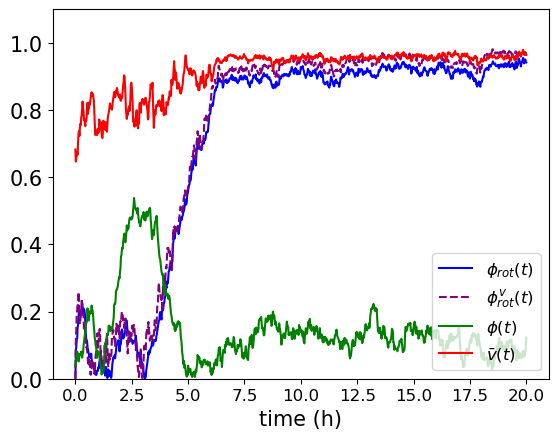}
    \subcaption{$\delta = 6.2 \; \rad\, \h^{-1}$, $\mu = 6.2 \; \rad\, \h^{-1}$}
    \label{fig:delta6.2_mu6.2}
    \end{subfigure}
    
    \caption{\reviewerA{(Disk domain) Cell configuration at final time $T=20\,\h$, with different values of $\delta$ and $\mu$. Parameters: $\Delta t = 10^{-2}\,\h$, $N= 160$, cell density $\rho = 0.707$. (Left panels) Polarities are represented by red vectors, velocities by orange vectors and the contact forces by edges between the cells where they are activated, colored according to the  magnitude of the contact force. (Right panels) Normalized global mean speed (in red), polarity rotation order parameter (in blue), \reviewerAA{velocity rotation order parameter (in purple),} and order parameter (in green) as functions of time. See also the corresponding \Mov{1}, \Mov{2} and \Mov{3}. }}
    \label{fig:simu}
\end{figure}

\begin{figure}
    \captionsetup[subfigure]{justification=centering}
    \centering
    \begin{subfigure}{0.32 \textwidth}
    \includegraphics[width = \textwidth]{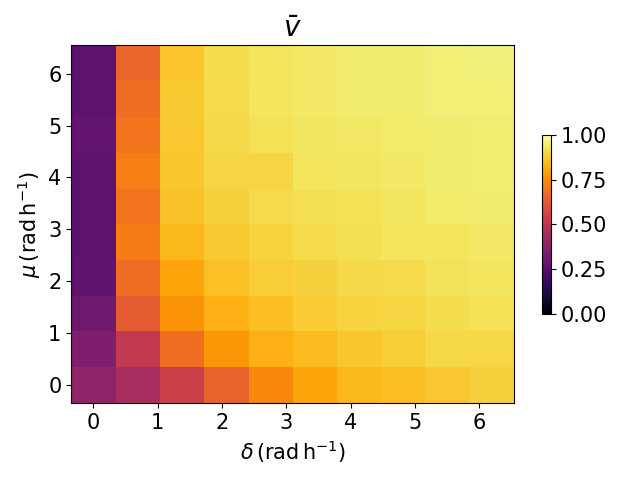}
    \subcaption{mean speed $\bar{v}$}
    \label{fig:phase_diagram_meanspeed}
    \end{subfigure}
    \begin{subfigure}{0.32 \textwidth}
    \includegraphics[width = \textwidth]{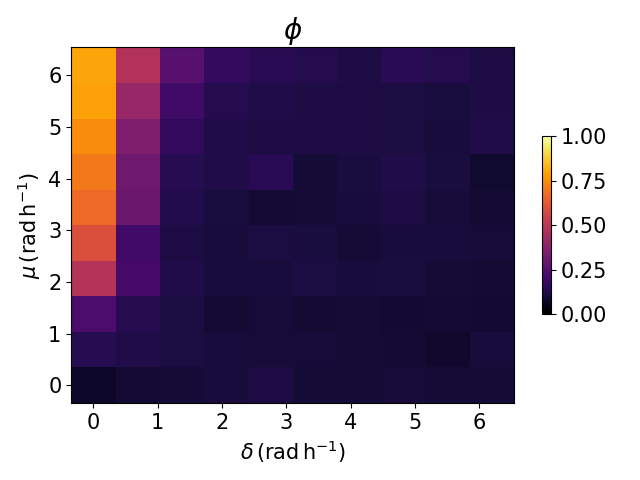}
    \subcaption{order parameter $\phi$}
    \label{fig:phase_diagram_phip}
    \end{subfigure}
    
    \begin{subfigure}{0.32 \textwidth}
    \centering
    \includegraphics[width = \textwidth]{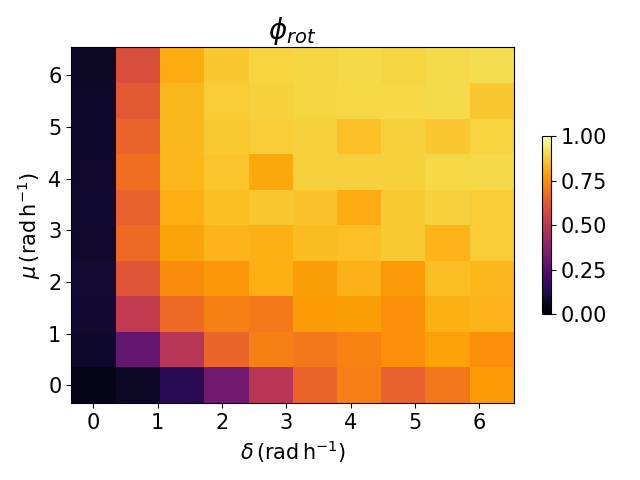}
    \subcaption{\reviewerAA{polarity} rotation\newline order parameter $\phirot$}
    \label{fig:phase_diagram_phirot}
    \end{subfigure}
    \begin{subfigure}{0.32 \textwidth}
    \centering
    \includegraphics[width = \textwidth]{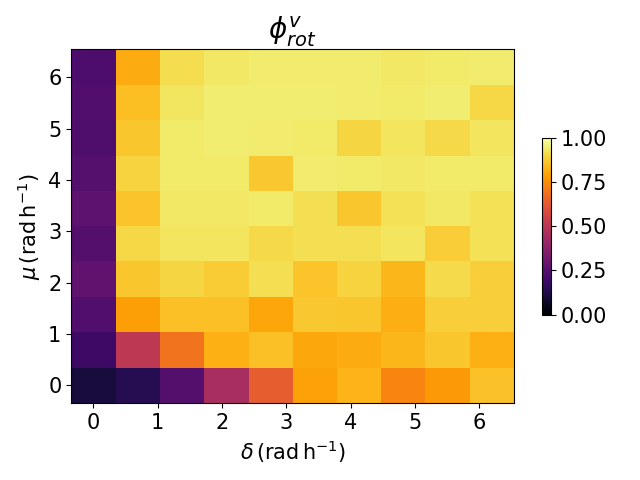}
    \subcaption{\reviewerAA{velocity} rotation\newline order parameter $\phirot^v$}
    \label{fig:phase_diagram_phirot_v}
    \end{subfigure}
    \caption{\reviewerA{(Disk domain) Mean speed, order parameter and rotation order parameters as functions of parameters $\delta$ and $\mu$. Each value is the average over $20$ simulations with parameters $\Delta t = 10^{-2} \h$, $N = 160$, cell density $\rho = 0.707$.}}
    \label{fig:phase_diagram}
\end{figure}

\begin{figure}
    \centering
    \hfill\begin{subfigure}{0.2 \textwidth}
    \includegraphics[width = \textwidth]{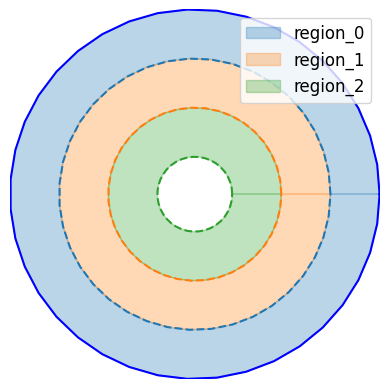}
    \subcaption{Regions}
    \end{subfigure}
    \begin{subfigure}{0.39\textwidth}
    \centering
    \includegraphics[width=0.8 \textwidth]{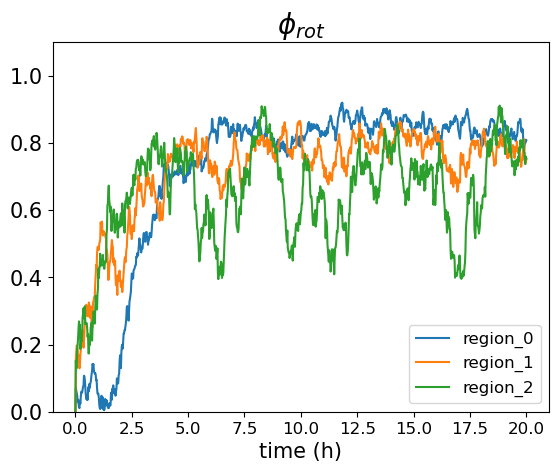}
    \subcaption{$\delta = 6.2 \; \rad\, \h^{-1}$, $\mu = 0 \; \rad\, \h^{-1}$}
    \label{fig:delta_stat}
    \end{subfigure}
    \begin{subfigure}{0.39\textwidth}
    \centering
    \includegraphics[width= 0.8\textwidth]{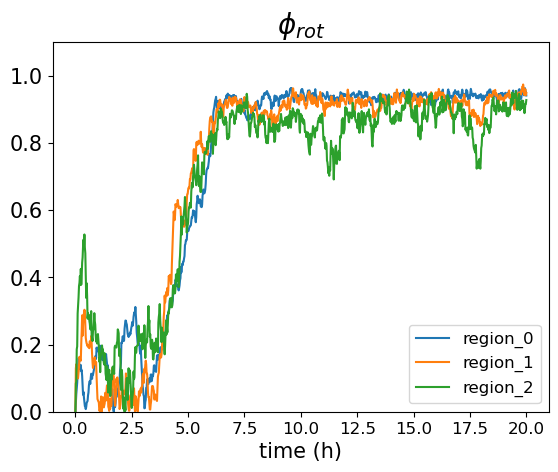}
    \subcaption{$\delta = 6.2 \; \rad\, \h^{-1}$, $\mu = 6.2 \; \rad\, \h^{-1}$}
    \label{fig:delta0_stat}
    \end{subfigure}

    \caption{\reviewerA{(Disk domain) \reviewerAA{Polarity} rotation order parameter as a function of time in different regions of the disk (depicted in the left panel), for two different values of the parameter $\mu$: $\mu = 0\ \rad\, \h^{-1}$ (middle) versus $\mu = 6.2\ \rad\, \h^{-1}$ (right). Parameters: $\Delta t = 10^{-2}\,\h$, $N= 160$, cell density $\rho = 0.707$. }}
    \label{fig:simu_stat_region}
\end{figure}

We now study the influence of density on these collective rotational movements in the disk domain. Figure~\ref{fig:circle_density_disc} describes the variation of the normalized mean speed and the rotation order parameters for densities ranging from $0.044$ up to $0.822$, which is the maximum possible density for this domain. We note that the rotation order parameters increase with density: the same behavior is observed for the Vicsek model with periodic condition. Note, however, that these rotation order parameters already take high values for low densities: this is certainly due to the boundary condition, which induces tangential velocities at the boundary of the domain. Regarding the mean speed, it remains rather close to $1$ and decreases very slightly with density. However, we note a larger variability for the maximal density, \reviewerAA{combined with a slight decrease of the velocity rotation parameter}. This results from a jamming effect which will be studied in further detail in the next section.

\begin{figure}[H]
    \centering
    \begin{subfigure}{0.9 \textwidth}
    \includegraphics[width=\textwidth]{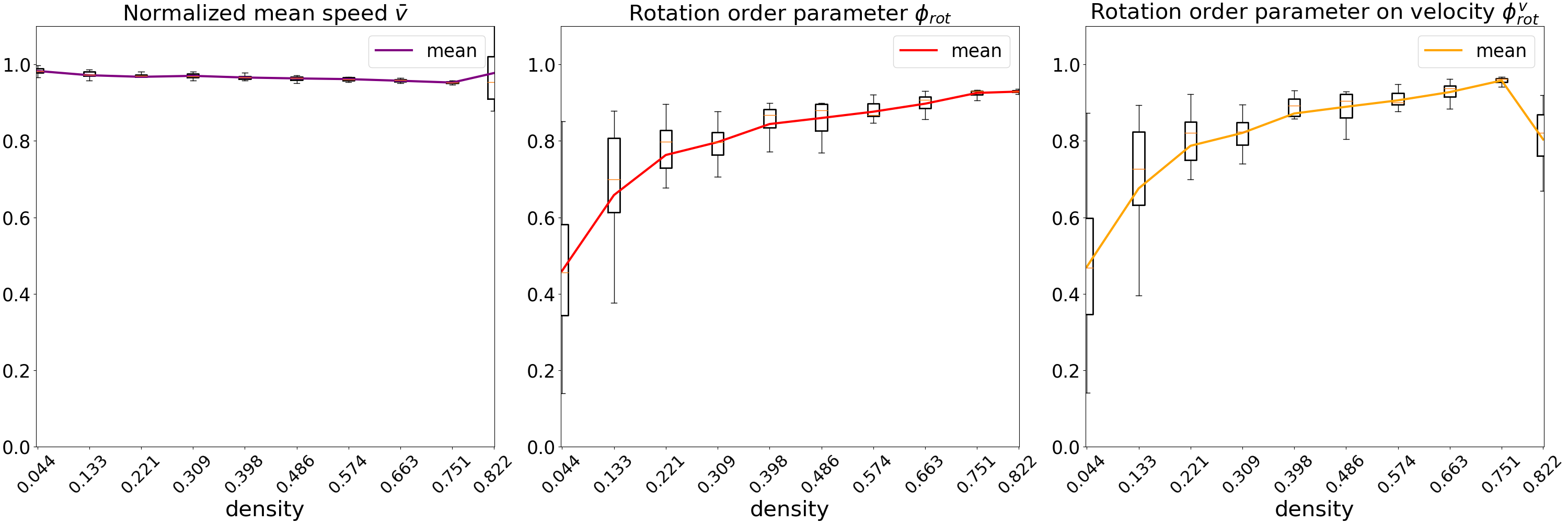}
    \subcaption{Disk domain}
    \label{fig:circle_density_disc}
    \end{subfigure}
    \vspace{0.2cm}
    
    \begin{subfigure}{0.9 \textwidth}
    \includegraphics[width=\textwidth]{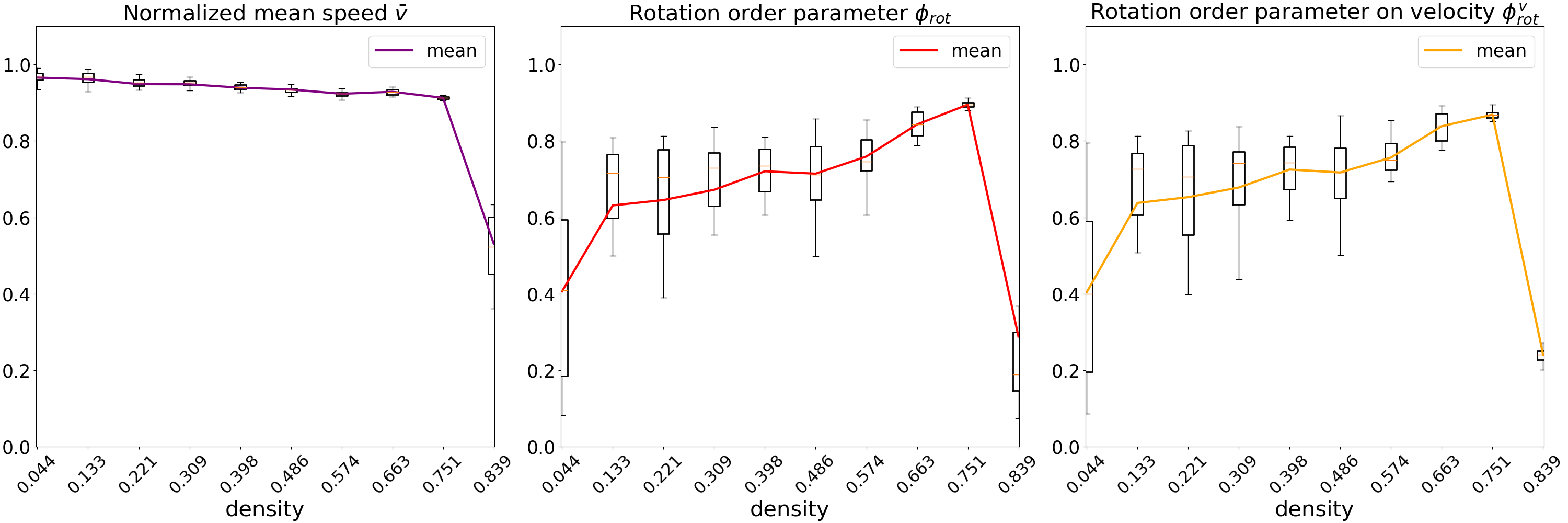}
    \subcaption{\reviewerA{Square domain}}
    \label{fig:circle_density_square}
    \end{subfigure}
    \caption{(Influence of the density and the domain shape) \reviewerA{Normalized global mean speed (left column), polarity rotation order parameter (middle column), \reviewerAA{velocity rotation order parameter (right column)} as functions of the density. Box plots have been obtained with $20$ simulations with parameters: $\Delta t = 10^{-2}\,\h$, $T= 20\,\h$. }}
    \label{fig:circle_density}
\end{figure}

\subsection{Influence of the domain shape and jamming effects}
\label{sec:jamming}

In this section, we investigate the influence of the domain shape on the cell dynamics, with a focus on the possible appearance of jamming. We consider two domains with the same area: a square domain of size $L = 200\, \microm$ and a disk of radius $L/\sqrt{\pi} \, \microm \approx 113 \,\microm$, and perform simulations up to time $T=20\,\h$ with a time step $\Delta t = 10^{-2}\,\h$. To measure the collective motion, we compute the normalized mean speed $\bar v$ and the rotation order parameters, $\phirot$ \reviewerAA{and $\phirot^v$,} defined in Eqs.~\eqref{eq:barv}-\eqref{eq:phirot} and averaged over the last $T/8 = 2.5\h$ of the simulations as computed for the polarity order parameter in Eq.~\eqref{eq:phip}.

\reviewerA{Figure~\ref{fig:circle_density_square} displays the variation of these \reviewerAA{three} indicators as the cell density increases for the square domain.} Again, to account for the stochastic aspects of the dynamics, $20$ different numerical simulations are performed for each panel. \reviewerA{Compared to Fig.~\ref{fig:circle_density_disc}, already commented in the previous section, we observe that the rotation order parameters are lower but still increase with density due to the interactions between cells. For the largest density, the three indicators decrease sharply, showing that the square domain refrains the rotational motion.}

\reviewerA{We present in Fig.~\ref{fig:domaincomparison} a similar analysis, with a focus on high densities.}
 When the domain is a disk (Fig.~\ref{fig:diskdomain}), we observe that \reviewerAA{the three} indicators remain very close to $1$ for the whole range of considered densities\reviewerAA{, with a slight decrease regarding the velocity rotation order parameter}. Indeed, since the diffusion parameter $D$ is small, collective motion is expected to occur. On the contrary, in the square domain (Fig.~\ref{fig:squaredomain}), \reviewerAA{the velocity rotation order parameter decreases across the entire density range,} while the mean speed and the polarity rotation order parameter decrease when the density is larger than $0.826$. We also note that the variability is larger above this threshold for these two quantities. \reviewerAA{These observations indicate that congestion effects prevent velocities from aligning with polarities more strongly as density increases and prevent the development of collective effects on polarity only beyond the density threshold.} 
 
 \reviewerA{Combining the results of Figs.~\ref{fig:domaincomparison} and ~\ref{fig:circle_density}, we infer that the optimal density for having a rotation motion in a square ranges between 0.75 and \reviewerAA{0.80}, whereas all densities larger than 0.75 lead to rotation in a disk. This difference can be explained by the fact that, in such a high-density regime, hexagonal lattice configurations form. While a solid rotation with a hexagonal lattice configuration is possible in the disk domain, this is no more the case in a square domain, where local lattice breaking is required to generate motion.}

\begin{figure}[h]
    \centering
    \begin{subfigure}{\textwidth}
    \includegraphics[width=\textwidth]{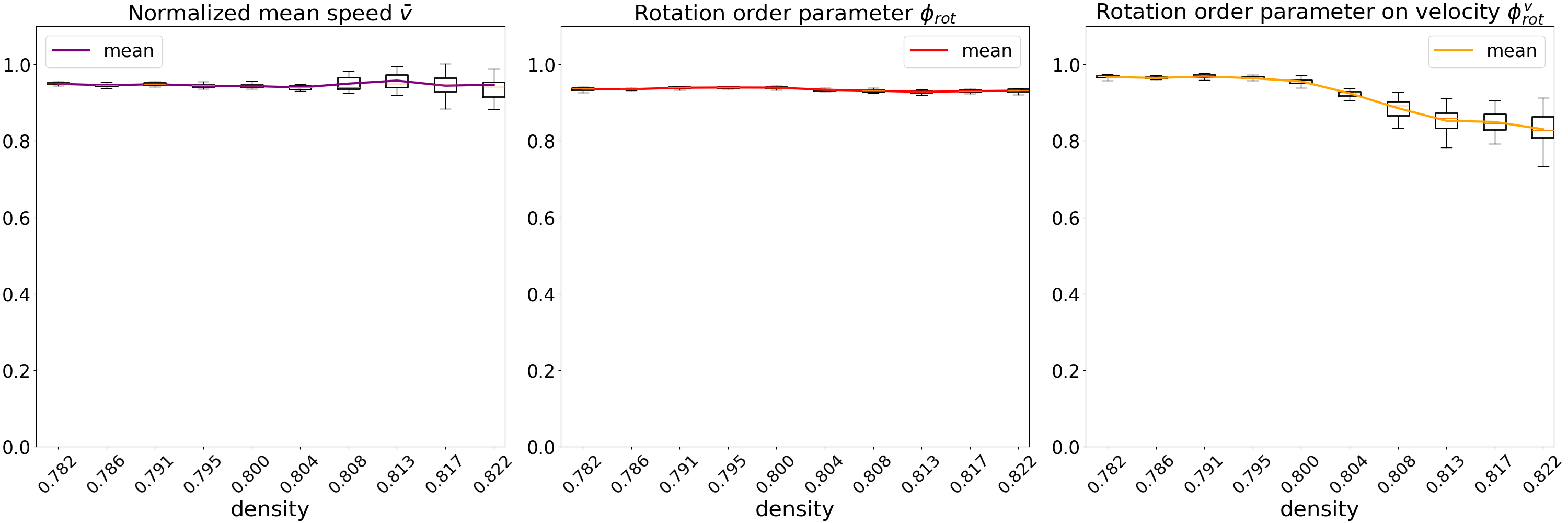}
    \subcaption{Disk domain}
    \label{fig:diskdomain}
    \end{subfigure}
    \vspace{0.2cm}
    
    \begin{subfigure}{\textwidth}
    \includegraphics[width=\textwidth]{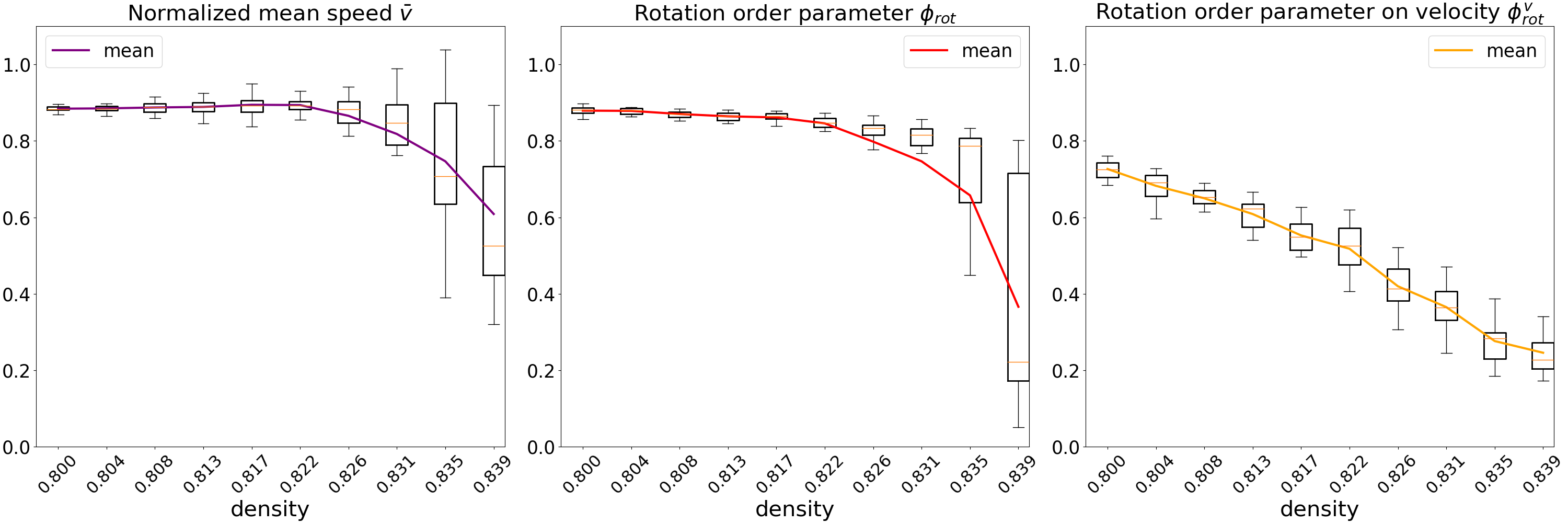}
    \subcaption{Square domain}
    \label{fig:squaredomain}
    \end{subfigure}

    \caption{\reviewerA{(Influence of the density and the domain shape - high densities) Normalized mean speed as a function of the density (left column), polarity rotation order parameter (middle column),
    \reviewerAA{velocity rotation order parameter (right column)}. Box plots have been obtained with $20$ simulations with parameters: $\Delta t = 10^{-2}\,\h$, $T= 20\,\h$.}}
    \label{fig:domaincomparison}
\end{figure}

\paragraph{Influence of obstacles.} To further assess the impact of the domain shape and topology on the collective motion, we add one or more obstacles inside the domain. Each obstacle is a disk of radius $R = 7.5\,\microm$ and could be seen as a non moving cell.

In Fig.~\ref{fig:square_1obstacle}, we display the three previous motion indicators when one obstacle is present in a square domain. Adding one obstacle in the center of the domain \reviewerA{(see Fig.~\ref{fig:square_obstacles_1})} seems to prevent the rotation dynamics. Indeed, the rotation order parameters are slightly smaller than in the case without obstacles 
\reviewerAA{(dashed lines)} for densities above $0.830$. We also observe that the mean speed has a wider variability for large densities. Fig.~\ref{fig:square_obstacles_1_de} shows that the rotation dynamics are further impaired when the obstacle is now located close to the boundary at distance $2R$ from the boundary: only one cell can pass between the obstacle and the border. \reviewerA{When the obstacle is placed near the corner (Fig.~\ref{fig:square_obstacles_1_coin}), the rotation parameters take values at the same level as in the case of a central obstacle: these two positions of the obstacle are the most compatible with rotational dynamics.}

We next consider a domain with four obstacles, first located in the middle of each side and  at a distance $2R$ to the boundary (Fig.~\ref{fig:square_obstacles_4}). We observe that both the mean speed and the rotation order parameters are much lower than in the case without obstacles 
\reviewerAA{(dotted lines)}
and with only one obstacle  on the side 
\reviewerAA{(dashed lines)}.
These obstacles have thus clearly introduced strong constraints on the global motion of the cells. In particular, the mean speed \reviewerAA{and the velocity rotation order parameters} are diminished on the full density range, which indicates that the motion is more jammed \reviewerAA{and velocities deviate a lot from the polarities directions due to the contact forces}. \reviewerA{When placing the four obstacles in the corner (Fig.~\ref{fig:square_obstacles_4coin}), we instead recover \reviewerAA{higher values of rotation order parameters} 
at least for densities lower than $0.819$. To further compare the dynamics in these two cases, Fig.~\ref{fig:square_obstacle_4_dynamic} shows the time evolution of the mean speed and the three order parameters for two representative simulations, one for each obstacles configurations.} \reviewerB{The corresponding movies,  \Mov{4} and \Mov{5}, are provided in the supplementary material.} \reviewerA{The number of cells is equal to $180$, corresponding to a density equal to $0.810$, for which a significant difference in results can be observed in Fig.~\ref{fig:square_4obstacles}. In Fig.~\ref{fig:square_obstacle_4_dynamic}, we observe significant variations in mean speed due to the frequent occurrence of blocked configurations: in much of the domain, the configuration is close to being a compact hexagonal lattice. This is also reflected in the histogram of inter-cell distances (right panels), where we can see that the first peaks are associated with distances of this type of lattice. As observed in the right panels, the appearance of rotation motion seems triggered by the local breaking of this lattice close to the corners.}

Altogether, the numerical results show that the shape and the topology of the simulation domain have a significant impact on the onset of collective movement of cells, translating in particular the ability of the model to capture the influence of the environment`s geometrical characteristics on the cellular dynamics. An ordered state is reproduced as expected in a circular domain, but also for a certain range of cellular densities in a square shaped domain, similar to experimental and numerical results in \cite{szaboetal2006}. Moreover, we highlight the complex role of obstacles: \reviewerA{obstacles placed in the corner of the domain slightly alter the emergence of collective rotation, whereas  obstacles on the sides favor the transition towards a jammed regime.}

\begin{figure}
    \centering
    \begin{subfigure}{\textwidth}
    \begin{minipage}[c]{0.13\textwidth}
    \centering
    \includegraphics[width=\textwidth]{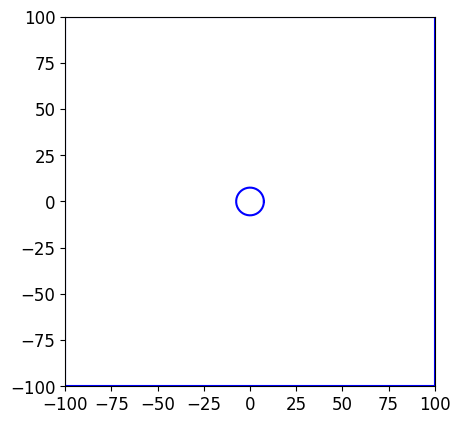}
    \end{minipage}
    \hfill
    \begin{minipage}[c]{0.84\textwidth}
    \centering
    \includegraphics[width=\linewidth]{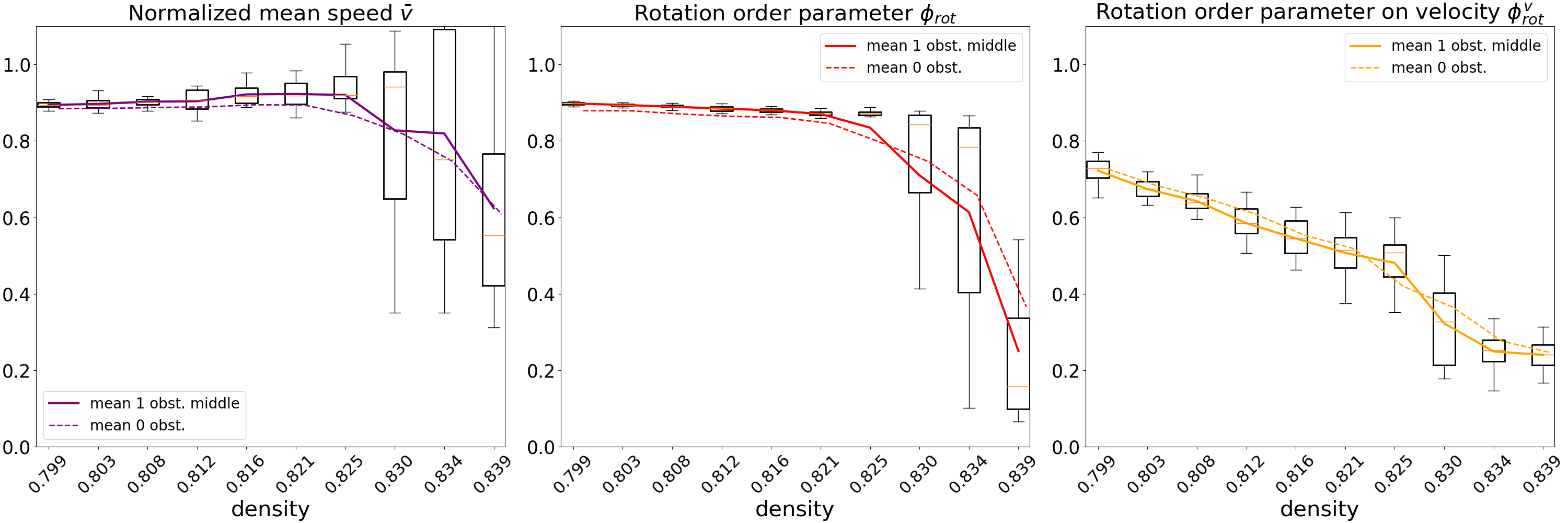}
    \end{minipage}
    \subcaption{1 obstacle in the middle}
    \label{fig:square_obstacles_1}
    \end{subfigure}
    \vspace{0.3cm}
    
    \begin{subfigure}{\textwidth}
    \begin{minipage}[c]{0.13\textwidth}
    \centering
    \includegraphics[width=\textwidth]{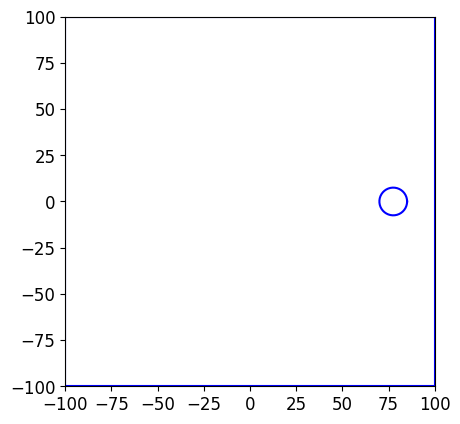}
    \end{minipage}
    \hfill
    \begin{minipage}[c]{0.84\textwidth}
    \centering
    \includegraphics[width=\linewidth]{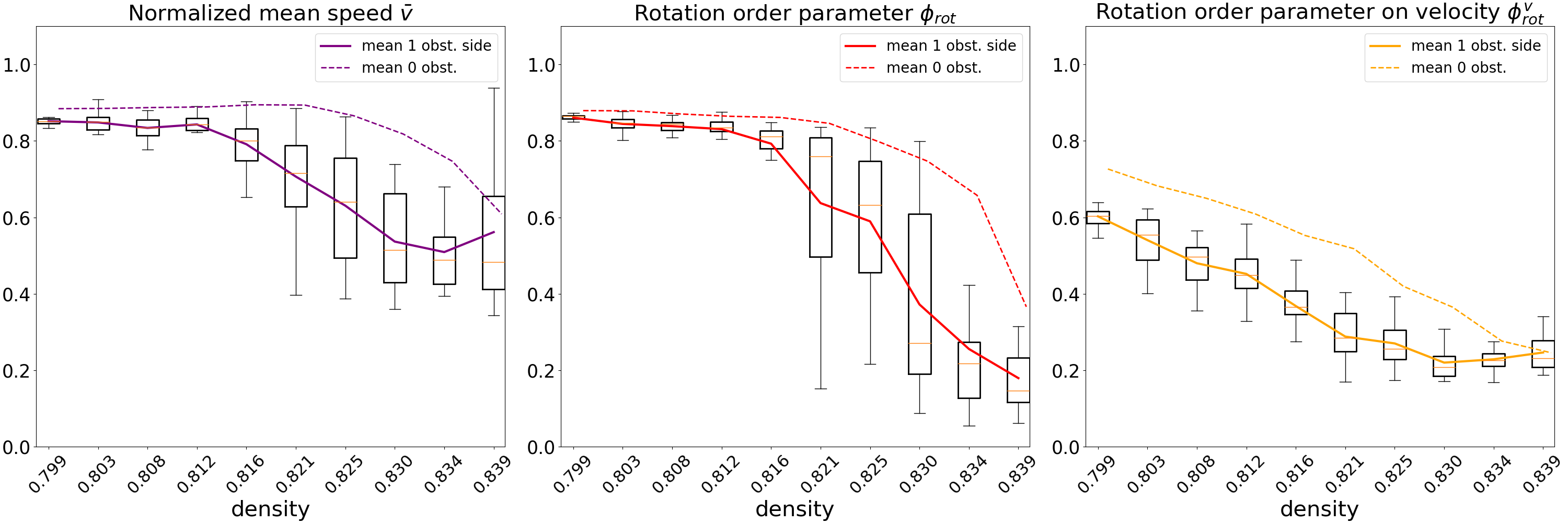}
    \end{minipage}
    \subcaption{1 obstacle on the side}
    \label{fig:square_obstacles_1_de}
    \end{subfigure}
    \vspace{0.3cm}
    
     \begin{subfigure}{\textwidth}
         \begin{minipage}[c]{0.13\textwidth}
    \centering
    \includegraphics[width=\textwidth]{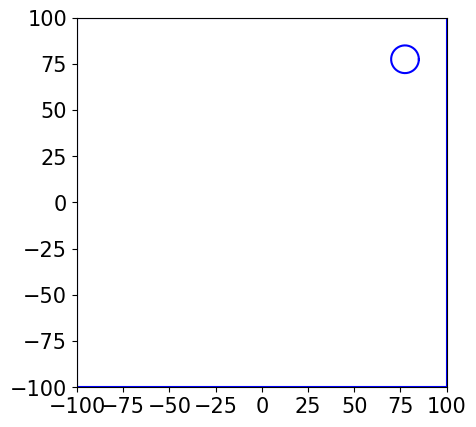}
    \end{minipage}
    \hfill
    \begin{minipage}[c]{0.84\textwidth}
    \centering
    \includegraphics[width=\linewidth]{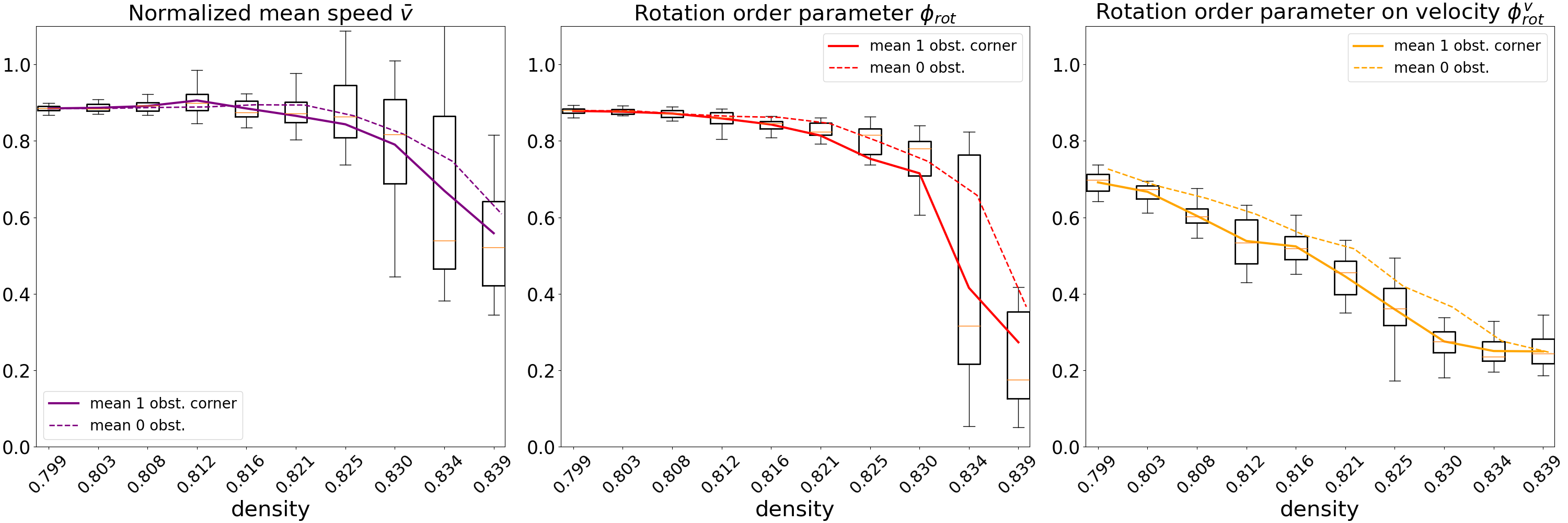}
    \end{minipage}
    \subcaption{\reviewerA{1 obstacle in the corner}}
    \label{fig:square_obstacles_1_coin}
    \end{subfigure}

    \caption{(One obstacle in a square domain) \reviewerA{Normalized mean speed (left column), polarity rotation order parameter (middle column), \reviewerAA{velocity rotation order parameter (right column)} as functions of the density. \reviewerAA{Solid lines: curves with obstacles. Dashed lines: curves without obstacle.} Box plots have been obtained with $20$ simulations with parameters: $\Delta t = 10^{-2}\, \h$, $T= 20\, \h$.}}
    \label{fig:square_1obstacle}
\end{figure}

\begin{figure}
    \centering
    \begin{subfigure}{\textwidth}
    \begin{minipage}[c]{0.13\textwidth}
    \centering
    \includegraphics[width=\textwidth]{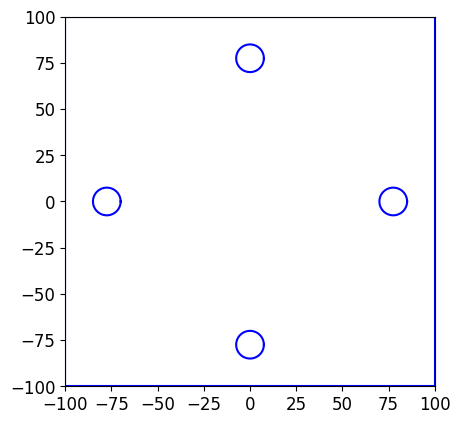}
    \end{minipage}
    \hfill
    \begin{minipage}[c]{0.84\textwidth}
    \centering
    \includegraphics[width=\linewidth]{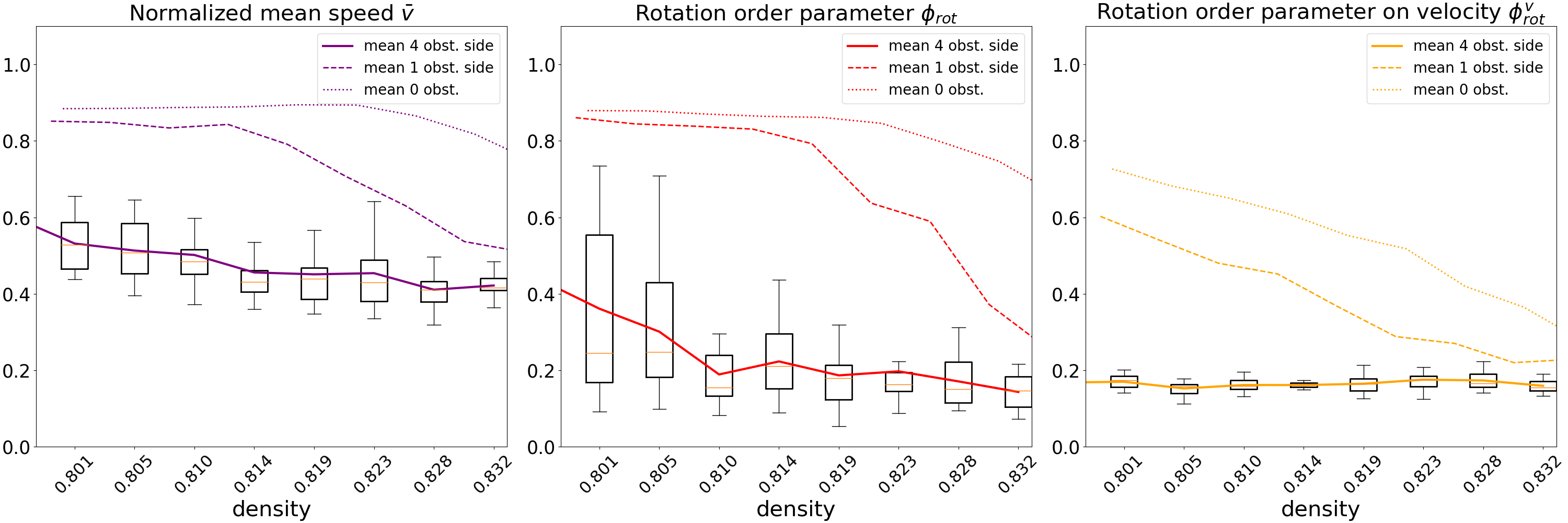}
    \end{minipage}
    \subcaption{4 obstacles on the sides}
    \label{fig:square_obstacles_4}
    \end{subfigure}
    \vspace{0.3cm}
    
    \begin{subfigure}{\textwidth}
    \begin{minipage}[c]{0.13\textwidth}
    \centering
    \includegraphics[width=\textwidth]{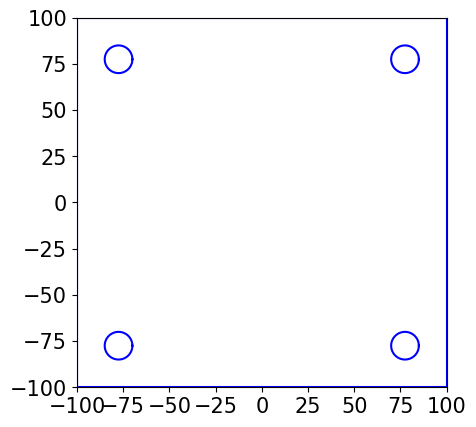}
    \end{minipage}
    \hfill
    \begin{minipage}[c]{0.84\textwidth}
    \centering
    \includegraphics[width=\linewidth]{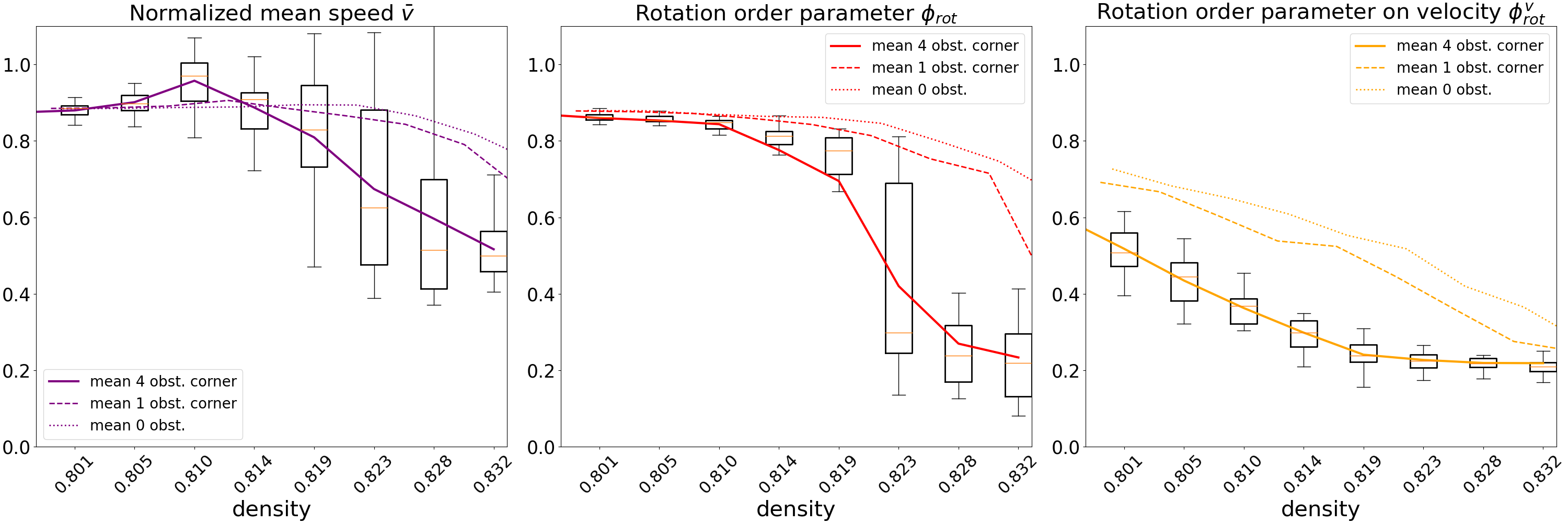}
    \end{minipage}
    \subcaption{\reviewerA{4 obstacles in the corners}}
    \label{fig:square_obstacles_4coin}
    \end{subfigure}
    \caption{(Four obstacles in a square domain) \reviewerA{Normalized mean speed (left column), polarity rotation order parameter (middle column), \reviewerAA{velocity rotation order parameter (right column)} as functions of the density. \reviewerAA{Solid lines: curves with obstacles. Dashed lines: curves with 1 obstacle. Dotted lines: curves without obstacle.} Box plots have been obtained with $20$ simulations with parameters: $\Delta t = 10^{-2}\, \h$, $T= 20\, \h$.}}
    \label{fig:square_4obstacles}
\end{figure}

\begin{figure}
    \centering
    \begin{subfigure}{\textwidth}
    \includegraphics[width=0.3\textwidth]{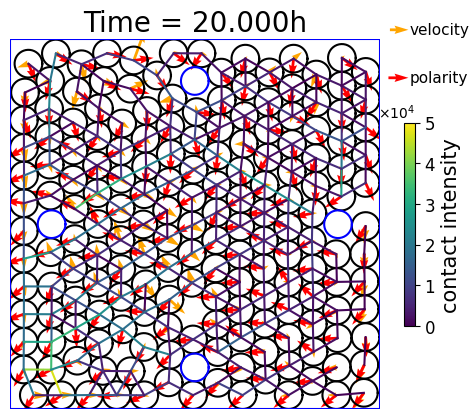}
    \includegraphics[width=0.3\textwidth]{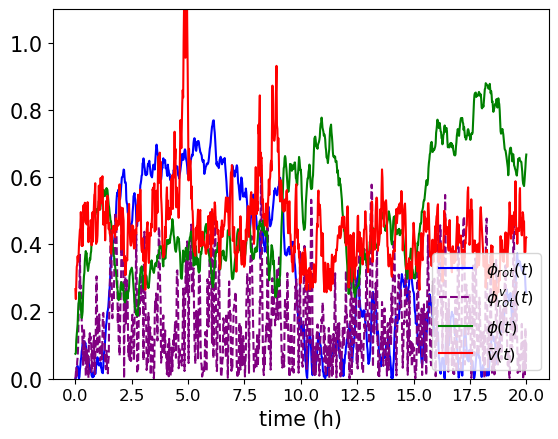}
    \includegraphics[width=0.33\textwidth]{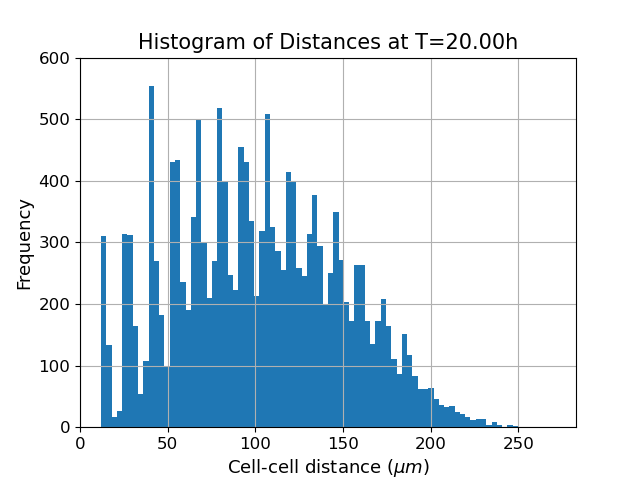}
    \subcaption{4 obstacles on the sides}
    \label{fig:4side}
    \end{subfigure}
    \begin{subfigure}{\textwidth}
    \includegraphics[width=0.3\textwidth]{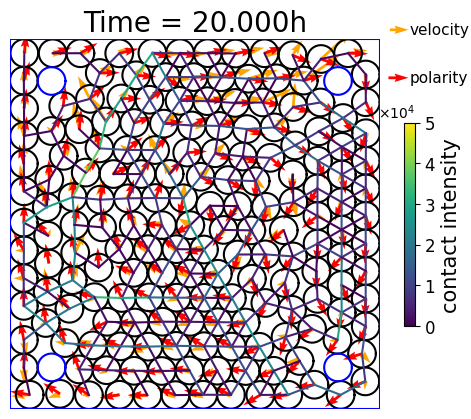}
    \includegraphics[width=0.3\textwidth]{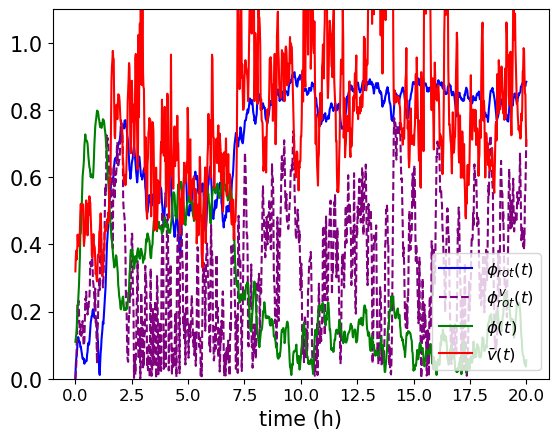}
    \includegraphics[width=0.33\textwidth]{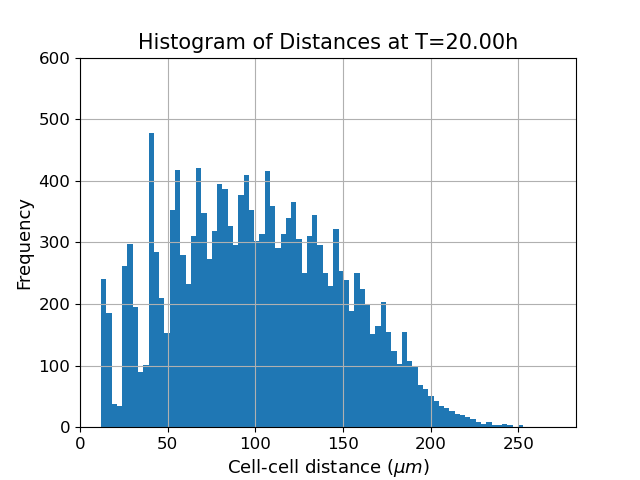}
    \subcaption{4 obstacles in the corners}
    \label{fig:4corner}
    \end{subfigure}
    \caption{\reviewerB{(Four obstacles in a square domain) Cell configuration at final time $T=20\,\h$ (left panel), order parameters during time (middle panel) and histogram of the cell-cell distance at final time. Parameters: $\Delta t = 10^{-2}\,\h$, $N= 180$, cell density $\rho = 0.810$. See also the corresponding \Mov{4} and \Mov{5}.}}
    \label{fig:square_obstacle_4_dynamic}
\end{figure}

\subsection{Influence of the soft attraction-repulsion force}

\label{part: Attraction-repulsion force}

In this section, we aim to investigate the impact of the soft repulsion force that we added in the present model to incorporate elastic interactions and to assess to which extent its variations modify the jamming properties of the system. To make attraction force effective, the attraction-repulsion radius $\Rintar$ is taken equal to $\Rintar = 21 \, \microm$, which is larger than twice the comfort radius.

We first display in Fig.~\ref{fig:hard} results obtained with parameters extracted from Table~\ref{tab:parameter}. We notice that many contacts occur between cells (left panel) and an overall rotational movement emerges (right panel). See also \Mov{6}. We then multiply the rigidity constant $\kappa$ by a factor $16$ and the inverse friction parameter $\gamma$ by a factor $10$ to significantly increase the impact of these interactions\reviewerA{, corresponding to a multiplication by a factor of 160 of the product $\gamma\kappa$}. The resulting cell configuration is very different, see Fig.~\ref{fig:smooth} (\Mov{7}): the effect of the comfort zone around each cell is clearly visible (left panel). Although the density of cells is quite high ($\rho = 0.850$ if we take into account the cell comfort zones), cells undergo a collective rotational motion (right panel). The low value of the order parameter $\phi$ indicates that cells are not fully aligned. This rotational movement may be triggered by the compressibility of the cells. On the contrary, if we remove the comfort zone by  setting $R_c$ to the cell radius $R_0$, we display in Fig.~\ref{fig:mix} (\Mov{8}) a new configuration, where cells act as if there were ``glued'' together and move collectively (left panel). Actually, cells are constantly moving from a border to the opposite one, which explains the oscillatory time evolution of the order parameter $\phi$ (right panel).

The numerical experiments thus demonstrate the effect of including a soft attraction-repulsion force in the model, which can significantly change the dynamics of the cells from a collective ``rotational'' movement in a more or less compact configuration to a ``translational'' collective movement in a packed form.

\begin{figure}
    \centering
    \begin{subfigure}{\textwidth}
    \centering
    \includegraphics[height=5cm]{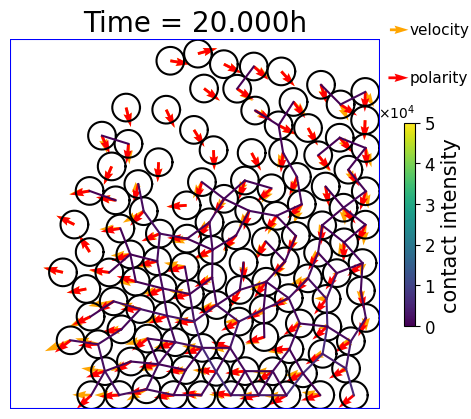}\qquad\qquad\includegraphics[height=5cm]{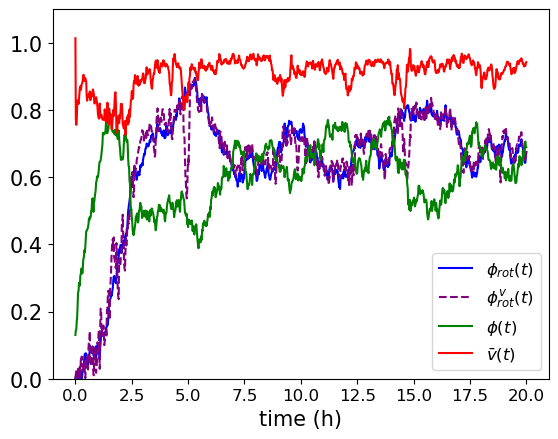}
    \subcaption{$R_c = 9.5\, \microm$, $\kappa= 10^4 \, \pN \, \microm^{-1}$, $\gamma=10^{-5} \, \pN^{-1} \h^{-1} \microm$, $\Rintar = 21 \, \microm$}
    \label{fig:hard}
    \end{subfigure}
    \begin{subfigure}{\textwidth}
    \centering
    \includegraphics[height=5cm]{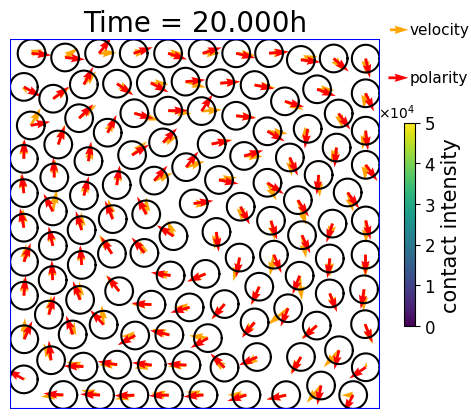}\qquad\qquad
    \includegraphics[height=5cm]{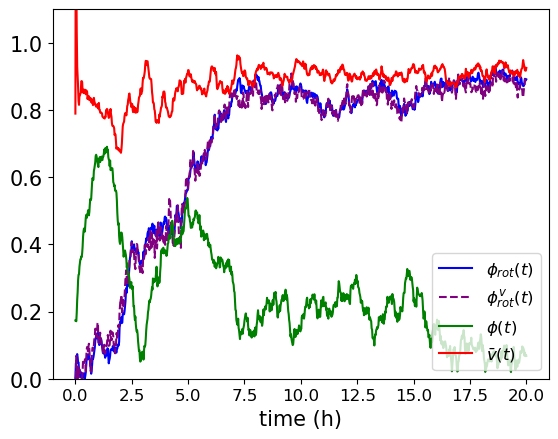}
    \subcaption{$R_c = 9.5\, \microm$, $\kappa= 16.10^4 \, \pN\, \microm^{-1}$, $\gamma=10^{-4} \, \pN^{-1} \h^{-1} \microm$, $\Rintar = 21 \, \microm$}
    \label{fig:smooth}
    \end{subfigure}
    \begin{subfigure}{\textwidth}
    \centering
    \includegraphics[height=5cm]{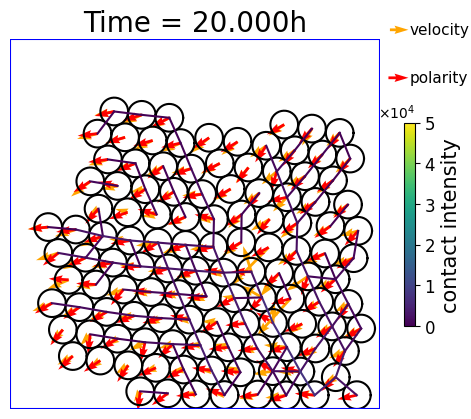}\qquad\qquad
    \includegraphics[height=5cm]{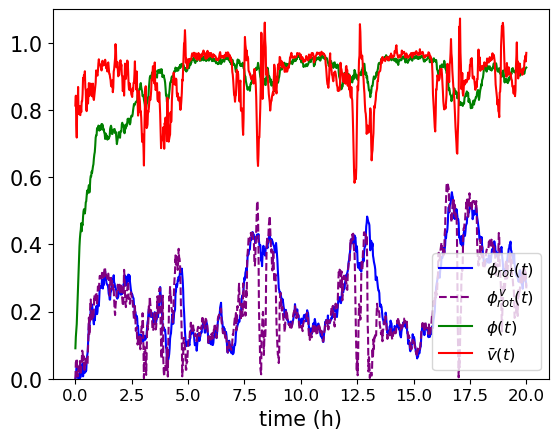}
    \subcaption{$R_c = 7.5\, \microm$, $\kappa= 16.10^4 \, \pN\, \microm^{-1}$, $\gamma=10^{-4} \, \pN^{-1} \h^{-1} \microm$, $\Rintar = 21 \, \microm$}
    \label{fig:mix}
    \end{subfigure}
    \caption{(Influence of the attraction-repulsion force) Cell configuration at final time $T=20\,\h$, (a) with default soft attraction-repulsion forces, (b) with increased soft attraction-repulsion forces, (c) with increased soft attraction but no soft repulsion forces. Parameters: $\Delta t = 10^{-2}\,\h$, $N= 120$, cell density $\rho = 0.530$. In panels (a) and (b), if the comfort zone is taken into account, cell density equals $\rho = 0.851$. Left column: polarities are represented by red vectors, velocities by orange vectors and the contact forces by edges between the cells where they are activated. See also comments caption Fig. 2. Right column: Normalized global mean speed (in red), \reviewerAA{polarity} rotation order parameter (in blue), \reviewerAA{velocity rotation order parameter (in purple),} and order parameter (in green) as functions of time. See the corresponding \Mov{6}, \Mov{7}, \Mov{8}. }
\end{figure}

\section{Conclusion and outlook}
\label{sec:conclusion}

We proposed in the present contribution a Vicsek-type model with contact forces to describe the dynamics of cells in a domain with wall-type boundaries. After studying the well-posedness of the \reviewerA{problem}, we introduced an adequate discrete version of the system and proposed a suitable numerical method to solve it. Finally, we evaluated the performance of our {\it in silico} approach through a systematic battery of numerical tests. Let us summarize their findings. First, order-disorder phase transition is naturally recovered when considering periodic boundary conditions and increasing the diffusion coefficient. Next, we showed that the velocity feedback on the polarity plays a crucial role to generate rotational movement in bounded domains. We equally demonstrated that, for large cell densities, the shape of the domain has a significant impact on the collective motion: rotational motion is easier in a disk than in a square. Adding an obstacle plays a complex role: it may or may not \reviewerA{prevent} the rotation movement to emerge, depending on its distance to the boundary and on the number of obstacles in the domain. Finally, by increasing the soft attraction-repulsion interactions, comfort zones can be effective and the compressibility of cells can also play a favorable role in triggering the rotational movement.

The perspectives of this work are multiple. First, the transition between flocking and elastic regimes would merit further investigation, as it can notably depend on the chosen attraction-repulsion potential. The parameters used in the simulations have been mostly taken from \cite{nature_phys}, where they have been calibrated to capture the {\em in vitro} dynamics of MDCK (Madin-Darby Canine Kidney) cells. Therefore, other parameter calibrations would be necessary to tackle elastic regimes observed in different biological systems. Additionally, more biological features could be incorporated, such as cell division or cell apoptosis. Finally, the analysis of the emerging rheological properties at the tissue scale is certainly the most challenging perspective.

\paragraph{Acknowledgements.} The authors warmly thank Romain Levayer and Daniel Riveline for their stimulating discussions. This research was funded by l’Agence Nationale de la Recherche (ANR), project
ANR-22-CE45-0028. 

\appendix

\section{Implementation and computational time}
\label{appendix:comput_time}
    
    \reviewerB{In this section, we provide information regarding the implementation of the model.}

    \reviewerB{The numerical method has been implemented in an in-house Python code. Each simulation starts with the initialization of the cell configuration. It is obtained as follows: the positions of the cells are randomly taken uniformly over the domain, the velocities are initially all of norm $c$, angles are uniformly distributed in $[0,2\pi)$ and the polarities are aligned with the velocities. A few steps of the algorithm achieve a configuration without overlapping cells. This sets the initial configuration. The most time demanding part of the code concerns the Uzawa algorithm, utilized to compute the contact forces. To optimize it, sparse matrix structures from \texttt{SciPy.sparse} have been used. The simulations for a single test were performed on a MacBook Pro M3, while those containing a statistical study were run in parallel on a computing server (a cluster of 6 nodes, each equipped with 2 AMD 64-core processors).} 
    
    \reviewerB{With regard to execution time, Table~\ref{tab:density} shows the results obtained for simulations in a square domain with a final time $T = 20$h and a time step $\Delta t = 10^{-2}$h. At a fixed density, increasing the number of cells and the size of the domain leads to a linear increase in computation time. We also provide the associated number of Uzawa iterations. As this number remains of a similar order of magnitude, the increase in execution time is solely due to the size of the matrices involved. It should be noted that the number of non-zero entries in these matrices corresponds to the number of contacts, which increases almost linearly with the number of cells. In Table~\ref{tab:lentgh}, we examine the increase in computational time as the density increases in a fixed domain. While the computation time increases almost linearly with the number of cells up to the density of 0.70, a sharp increase is observed for a density equal to 0.839 both for initialization and simulation times. This results from the increase in the number of Uzawa iterations due to congested configurations.}

    \begin{table}
    \centering
    \begin{tabular}{c|cccccc}
         number of cells & 50 & 100 & 200 & 300 & 400 & 500 \\
         length $(\mu m)$ & 112 & 159 & 225 & 275 & 318 & 355 \\
         \hline
         initialization time (s) & 0.36 & 0.86 & 1.81 & 3.13 & 8.14 & 12.85\\
         simulation time (s) & 26.67 & 42.59 & 86.00 & 131.3 & 194.35 & 258.85\\
         nb of Uzawa iterations & 34750 &30391 & 29929 & 26417 & 27607 & 26755 
    \end{tabular}
    \caption{\reviewerB{(Computational time - fixed density) Initialization time, simulation time  and total number of Uzawa iterations for simulations in a square domain with increasing lengths and a fixed density of cells $\rho = 0.7$, final time $T = 20$h and time step $\Delta t = 10^{-2}$h.}}
    \label{tab:density}
\end{table}

\begin{table}
    \centering
    \begin{tabular}{c|ccccc}
         number of cells &  40 & 80 & 120 & 160 & 190\\
         density ($\rho$) & 0.177 & 0.353 & 0.530 & 0.707 & 0.839 \\
         \hline
         initialization time (s) & 0.08 & 0.15 & 0.25 & 0.89 & 25.45 \\
         simulation time (s) & 17.66 & 33.74 & 46.61 & 65.16 & 139.03 \\
         nb of Uzawa iterations & 30367 & 30281 & 29159 & 27629 & 90445
    \end{tabular}
    \caption{ \reviewerB{(Computational time - increasing density) Initialization time, simulation time  and total number of Uzawa iterations for simulations in a square domains with a fixed size (200 $\mu$m) and increasing number of cells, final time $T = 20$h and time step $\Delta t = 10^{-2}$h.}}
    \label{tab:lentgh}
\end{table}

\section{List of supplementary videos}
\label{app:videos}

Here we present the list of additional videos, which can be accessed by following this link: \url{https://seafile.unistra.fr/d/472a60ad9ce546e6b80b/}. For each simulation, the final time is $20$h, which corresponds to $2000$ iterations. In the videos, polarities are represented by red vectors, velocities by orange vectors, and contact forces, when activated, by edges between cells. The edges are colored according to the magnitude of the contact forces.

\paragraph{Influence of the velocity feedback and polarity alignment}

\begin{itemize}
    \item \Mov{1} (\verb|circle_160_delta0_mu6.2.mov|) is the movie corresponding to Figure~\ref{fig:delta0_mu6.2}, with $N = 160$ cells and deactivated velocity feedback ($\delta = 0 \; \rad\, \h^{-1}, \mu = 6.2 \; \rad\, \h^{-1} $). 
    \item \Mov{2} (\verb|circle_160_delta6.2_mu0.mov|) is the movie corresponding to Figure~\ref{fig:delta6.2_mu0}, with $N = 160$ cells with deactivated polarity alignment.   ($\delta = 6.2 \; \rad\, \h^{-1}, \mu = 0 \; \rad\, \h^{-1} $). 
    \item \Mov{3} (\verb|circle_160_delta6.2_mu6.2.mov|) is the movie corresponding to Figure~\ref{fig:delta6.2_mu6.2}, with $N = 160$ cells with activated velocity feedback and polarity alignment.  ($\delta = 6.2 \; \rad\, \h^{-1}, \mu = 6.2 \; \rad\, \h^{-1} $). 
\end{itemize}

\paragraph{Influence of the position of obstacles}
\begin{itemize}
    \item \Mov{4} (\verb|square_180_4obstacles_middle.mov|) \reviewerB{is the movie corresponding to Figure~\ref{fig:4side}, with $N=180$ cells and $4$ obstacles added in the middle of each side. The parameters are those from Table~\ref{tab:parameter}.}
    \item \Mov{5} (\verb|square_180_4obstacles_corner.mov|) \reviewerB{is the movie corresponding to Figure~\ref{fig:4corner}, with $N=180$ cells and $4$ obstacles added in each corner. The parameters are those from Table~\ref{tab:parameter}.}
\end{itemize}

\paragraph{Influence of the soft attraction-repulsion force}
\begin{itemize}
    \item \Mov{6} (\verb|square_120_default_parameters.mov|) is the movie corresponding to Figure~\ref{fig:hard}, with $N = 120$ cells and default soft attraction-repulsion force. The parameters are those from Table~\ref{tab:parameter}, but with $\Rintar = 21 \, \microm$. 
    \item \Mov{7} (\verb|square_120_strong_attraction_repulsion.mov|) is the movie corresponding to Figure~\ref{fig:smooth}, with $N = 120$ cells and increased soft attraction-repulsion force ($\kappa= 16.10^4 \, \pN\, \microm^{-1}$, $\gamma=10^{-4} \, \pN^{-1} \h^{-1} \microm$, $\Rintar = 21 \, \microm$).
    \item \Mov{8} (\verb|square_120_strong_attraction.mov|) is the movie corresponding to Figure~\ref{fig:mix}, with $N = 120$ cells, increased soft attraction but no soft repulsion force ($\kappa= 16.10^4 \, \pN\, \microm^{-1}$, $\gamma=10^{-4} \, \pN^{-1} \h^{-1} \microm$, $R_c = 7.5\, \microm$, $\Rintar = 21 \, \microm$).
\end{itemize}

\bibliographystyle{plain}
\bibliography{biblio}

\end{document}